\documentclass{article}

\usepackage{amsmath}
\usepackage{amsthm}
\usepackage{amssymb}
\usepackage{amscd}
\usepackage[dvipdfmx]{graphicx}

\allowdisplaybreaks[4]
\theoremstyle{plain}
\newtheorem{theo}{Theorem}[section]
\newtheorem{prop}[theo]{Proposition}
\newtheorem{lemm}[theo]{Lemma}
\newtheorem{cor}[theo]{Cororally}

\theoremstyle{definition}
\newtheorem{defi}[theo]{Definition}
\newtheorem{exa}[theo]{Example}

\newtheorem{rmk}[theo]{Remark}

\newcommand{\identitymap}[1]{\mathrm{id}_{#1}} 
\numberwithin{equation}{section}

\newcommand{\vect}[1]{#1}

\newcommand{\circg}[1]{\langle #1 \rangle}
\newcommand{\Res}[2]{\mathrm{Res}^{#1}_{#2}}

\newcommand{\Mobius}{M\"{o}bius}

\begin{document}
\pagestyle{empty}

\author{Tomoyuki Tamura\\ Graduate school of Mathematics\\ Kyushu University \\ 2017}
\title{Symmetric powers of permutation representations of finite groups and primitive colorings on polyhedrons}
\date{}
\maketitle

\pagestyle{plain}

\begin{abstract}
In this paper\footnote[1]{Keywords and Phrases; finite group; Burnside ring: $\lambda$-ring; group action, symmetric powers representation, necklace polynomial.},  we define a set which has a finite group action and is generated by a finite color set, a set which has a finite group action, and a subset of the set of non negative integers. we state its properties to apply one of solution of the following two problems, respectively. First, we calculate the generating function of the character of symmetric powers of permutation representation associated with a set which has a finite group action. Second, we calculate the number of primitive colorings on some objects of polyhedrons. It is a generalization of the calculation of the number of primitive necklaces by N.Metropolis and G-C.Rota. 
\end{abstract}

\section{Introduction}
Let $G$ be a finite group. In this paper, we will discuss actions of $G$, which will be prepared in \S 2. In \S 3.1, we will define the set ``$|A|$-colored $N$-nested $G$-set", written by $J_{N,A,a}(X)$. Moreover, we will define the degree of an element of $J_{N,A,a}(X)$ and decompose $J_{N,A,a}(X)$ into the disjoint union of $G$-set $J^n_{N,A,a}(X)$, which is the set of $f\in J_{N,A,a}(X)$ such that $\deg(f)=n$ holds.

We will also define a power series $\varphi_{H,t}(J_{N,A,a}(X))$ by
\[ \varphi_{H,t}(J_{N,A,a}(X))=\sum_{n=0}^{\infty}\varphi_H(J^n_{N,A,a}(X))t^n\]
for any subgroup $H$ of $G$. 

Note that the definition of $J_{N,A,a}(X)$ depends on a finite set $A$, an element $a\in A$, a non-empty subset $N\subset\mathbb{N}\cup\{0\}$ and a $G$-set $X$. In \S 3.2 and \S 3.3, we will discuss the following two problems with $J_{N,A,a}(X)$ and the power series $\varphi_{H,t}(J_{N,A,a}(X))$:
\begin{itemize}
\item[{\rm (1)}] Calculation of characters of exterior powers of representations with a character of symmetric powers.
\item[{\rm (2)}] Calculation of the number of primitive colorings on some objects of polyhedrons.
\end{itemize}
In \S 1.1 and \S 1.2, we outline problems (1) and (2).

\subsection{Symmetric powers of representations}
For given representation $\rho:G\rightarrow GL(V)$ with finite dimension over the complex field $\mathbb{C}$ and integer $i\geq 0$, we can define the $i$-th exterior power of the representation $\rho$ and $i$-th symmetric power of the representation $\rho$ by
\begin{eqnarray*}
&{}&\Lambda^i\rho:G\rightarrow GL(\textstyle\bigwedge^i(V)),\\
&{}&\Lambda^i\rho(g)(v_1\wedge\cdots\wedge v_i):=(\rho(g)v_1)\wedge\cdots\wedge (\rho(g)v_i),\\
&{}&S^i\rho:G\rightarrow GL(\textstyle S^i(V)),\\
&{}&S^i\rho(g)(v_1\cdots v_i):=(\rho(g)v_1)\cdots (\rho(g)v_i)
\end{eqnarray*}
for any integer $i\geq 0$ where $g\in G$ and $v_1,\dots,\ v_i\in V$. We will apply $\varphi_{H,t}(J_{N,A,a}(X))$ to calculate the generating function of the character of $i$-th symmetric powers of permutation representation associated with $X$.  We consider this calculation to find how to calculate exterior powers of representations. This problem was raised in \cite{Knu}

First, we discuss $\lambda$-rings and relation between $\lambda$-rings and finite groups. A pre $\lambda$-ring is a commutative ring with operations $\lambda^n:R\rightarrow R$, $n=0,1,2,\dots,$ such that $\lambda^0(r)=1$, $\lambda^1(r)=r $ and $\lambda^n(r+s)=\sum_{i+j=n}\lambda^i(r)\lambda^j(s)$ hold for any $r\in R$ and integer $n \geq 0$. For any finite group $G$, the set $CF(G)$, which  is the set of all class function $f:G\rightarrow\mathbb{C}$ satisfying $f(g^{-1}xg)=f(x)$ for any $x,g\in G$, has a commutative $\mathbb{Q}$-algebra structure via addition, multiplication, scalar of $\mathbb{C}$ and a structure of pre-$\lambda$-ring\footnote[1]{In addition, Knutson said that the $CF(G)$ becomes a $\lambda$-ring. In this paper we will not state about $\lambda$-rings. see \cite{Knu}.} whose $n$-th Adams operations $\psi^n$ satisfies $\psi^n(f)(g):=f(g^n) $for any $f\in CF(G)$, $g\in G$ and integer $n\geq 1$. Knutson showed in \cite[p.84]{Knu} that the character of $\textstyle\Lambda^i(V)$ is equal to $\lambda^i(\chi)$ for any $i\geq 0$ where $\chi$ is the character of a representation $\rho$ and $\lambda^i$ is a $\lambda$-operation of $CF(G)$. 

In \S 2.3, we will introduce symmetric powers operations $S^n$, $n=0,1,2,\ldots$ on $\lambda$-rings, and prove that $S^n(\chi)$ is the character of the $n$-th symmetric powers of the representation whose character is $\chi$. By the definition of symmetric powers operations, the calculating of the generating function character $\lambda^n(\chi)$ is equivalent to the calculating of the generating function of $S^n(\chi)$.

In \S 3.2, we will discuss the case of $N=\mathbb{N}\cup\{0\}$ and $|A|=2$. We will show that the $G$-set $J_{N,A,a}(X)$ is isomorphic to the symmetric algebra of $X$ as a $G$-set. The symmetric algebra of $X$ was defined in \cite[\S 2.13]{DS}. Furthermore, for any $g\in G$ we will show that the generating function $\varphi_{\circg{g},t}(J_{N,A,a}(X))$ is equal to the generating function of $S^n(\chi)(g)$ where $\chi$ is the permutation character associated with $X$. By the definition of symmetric powers operation on arbitrary $\lambda$-rings and the result of \cite{Knu}, The calculation of the generating function of $S^n(\chi)(g)$ is equivalent to the calculation of the generating function of $\lambda^n(\chi)(g)$. Hence, the result of \S 3.2 will give one of calculation method of the character of exterior powers of representations.

\subsection{Primitive colorings on polyhedrons}

In this paper, we denote the group with one element by $C_1$. For any subgroup $V$ of $G$, we denote by $\mu_{H}(X)$ the number of orbits in $X$ which is isomorphic to $G/H$ (This notation will be defined again in \S 2).

 In \S 3.3, we will discuss the case of $N=\{0, 1\}$. We will show that the set $J_{N,A,a}(X)$ is identified with the set of maps $\iota :X\rightarrow A$, and the set $J^n_{N,A,a}(X)$ is identified with the set of maps $\iota :X\rightarrow A$ such that $|\{x \in X\mid \iota(x)\neq a\}|=n$ holds. We will introduce the method for calculating $\mu_{C_1}(J_{N,A,a}(X))$ and will introduce the method for calculating the generating function of $\mu_{C_1}(J^n_{N,A,a}(X))$.

Now, we discuss the meaning to calculate $\mu_{C_1}(J_{N,A,a}(X))$. We assume that $X$ is the set of objects on polyhedrons (vertices, edges), and $G$ is a rotation group of $X$ and $A$ is a finite color set. A primitive coloring on $X$ of $G$ with $A$ is a coloring on $X$ with $A$ which is asymmetric under rotation of $G$. More detail, see Definition \ref{color}.

We can regard the set $J^n_{N,A,a}(X)$ as the set of colorings on $X$ which has $(|X|-n)$ times of $a$ colored, and the set $J_{N,A,a}(X)$ as the set of all colorings on $X$ with color of $A$. Moreover, $\mu_{C_1}(J^n_{N,A,a}(X))$ is the number of primitive colorings which has $(|X|-n)$ times of $a$ colored,
 and $\mu_{C_1}(J^n_{N,A,a}(X))$ provides the number of primitive colorings.

For the these colorings, Metropolis and Rota \cite{MR} considered the number of primitive necklaces. Given a set of colors $A$, a necklace is a result of placing $n$ colored beads around a circle. A necklace which is asymmetric under rotation is said to be primitive.

For example, we consider when $n=6$ and $A=\{0,1\}$. Figure 1 is an example of primitive. However, Figure 2 is not primitive because it has a symmetry under $\dfrac{2\pi}{3}$ rotation.

\begin{figure}[h]
	\centering \includegraphics[width=2.5cm,height=2.5cm]{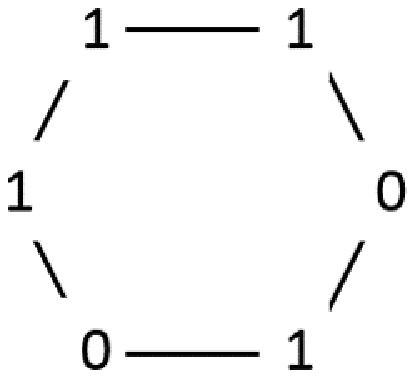}$\quad\quad\quad\quad\quad\ \ \ \ \ $
	\includegraphics[width=2.5cm,height=2.5cm]{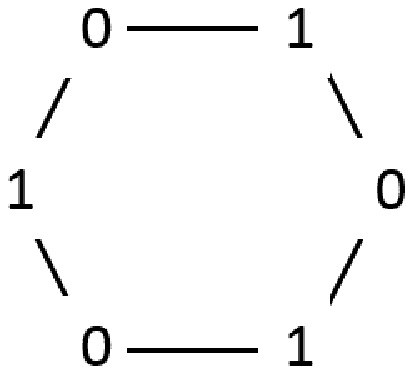}\\
	\centering Figure 1: primitive necklace$\quad\quad$ Figure 2: non-primitive necklace 
\end{figure}

The number of primitive necklaces $M(k,n)$ is computed by the following formula, 
\begin{eqnarray}\label{necklace}
M(k,n)=\dfrac{1}{n}\sum_{d\mid n}\mu\Big(\dfrac{n}{d}\Big)k^d,
\end{eqnarray}
where $k$ is the cardinality of $A$. Note that $M(k,n)$ is a polynomial with an indeterminate variable $k$. 

For each integer $n\geq 1$, let $X_n$ be the set of all vertices of a regular $n$-gon and let $C_n$ be the cyclic group whose cardinality is $n$. Then, we can regard the set $X_n$ as a $C_n$ set, and hence, we will show that $\mu_{C_1}(J_{N,A,a}(X_n))$ is the number of asymmetric $k$-colorings of vertices of the $n$-regular polygons, which also coincides with the necklace polynomial $M(|A|,n)$.

Metropolis and Rota \cite{MR} showed the following formula: 
\begin{eqnarray}\label{necklace2}
\dfrac{1}{1-kx}=\prod_{n=1}^{\infty}\Big(\dfrac{1}{1-x^n}\Big)^{M(k,n)}
\end{eqnarray}
for any integer $k\geq 1$, as combinatorial theory. The formula (\ref{necklace2}) is called a cyclotomic identity.

In addition, Metropolis and Rota showed the following identities for necklace polynomials:
\begin{eqnarray}
\label{necklaceM} M(k_1k_2,n)&=&\sum_{[i,j]=n}(i,j)M(k_1,i)M(k_2,j),\\
\label{necklaceF} M(k^r,n)&=&\sum_{[j,r]=nr}\dfrac{j}{n}M(k,j)
\end{eqnarray}
for any integers $k_1, k_2, n\geq 1$ and $r\geq 1$. Identities (\ref{necklaceM}) and (\ref{necklaceF}) are origin of the multiplication of necklace rings and Frobenius operations. In \S 3.4, we generalize identities (\ref{necklaceM}) and (\ref{necklaceF}) to (\ref{GeneM}) and (\ref{GeneF}) with $\mu_{V}(J_{N,A,a}(X))$ for all subgroups $V$ of $G$. We consider a regular $n$-gon in (\ref{GeneM}) or (\ref{GeneF}), and hence, we will obtain the same result of (\ref{necklaceM}) or (\ref{necklaceF}), respectively.

In \S 4, we consider $J_{N,A,a}(X)$ when
\begin{itemize}
\item [{\rm (i)}] The set $X$ is the set of all vertices of a regular $n$-prism and $G$ is a dihedral group $D_{n}$,
\item [{\rm (ii)}] The set $X$ is the set of all vertices of a regular $n$-gon and $G$ is a dihedral group $D_{n}$,
\end{itemize}
as examples.

\begin{rmk}
Suppose that $X$ is isomorphic to $G/C_1$. Note that for any subgroup $V$ of $G$, $\mu_{V}(J_{N,A,a}(X))$ is equal to the polynomial $M_G(k,V)$, which was introduced in \cite{Oh1}, \cite{Oh2}, where $k$ is the cardinality of $A$. In \cite{Oh1}, $k$ belongs to some $\lambda$-rings, and Oh used the polynomial $M_G(k,V)$ to define the ring homomorphism, called exponential map. In \cite{Oh2}, $k$ belongs to the set of all integers $\mathbb{Z}$, and Oh generalized identities (\ref{necklaceM}) and (\ref{necklaceF}) for $M_G(k,V)$.
\end{rmk} 

Here, we denote the following notations in this paper.
\begin{itemize}
\item[(i)] Let $\mathbb{N}$, $\mathbb{Z}$ or $\mathbb{Q}$ be the set of all positive integers, integers or rational integers, respectively. 
\item [(ii)]For any integers $i$ and $j$, a symbol $i\mid j$ stands for that $i$ divides $j$ and $i\nmid j$ stands for that $i$ does not divide $j$.
\item [(iii)]For any integers $i$ and $j$, a symbol $(i,j)$ stands for the greatest common divisor of $i$ and $j$, and a symbol $[i,j]$ stands for the least common multiple of $i$ and $j$.
\item [(iv)]We denote the\ \Mobius\ function by $\mu$.
\item [(v)]For any set $X$, we denote the identity map on $X$ by $\identitymap{X}$ defined by $\identitymap{X}(x)=x$ for any $x\in X$.
\item [(vi)] We denote the unit element of a finite group by $1$.
\item [(vii)] For each element $g$ of a finite group $G$, we denote by $\circg{g}$ the subgroup of $G$ generated by $g\in G$, and by $O(g)$ the order of $g$. We define the exponent of $G$ by the least common multiple of all $O(g)$'s for $g\in G$.
\item [(viii)] We assume that every ring and semiring have the unit element, written by $1$.
\end{itemize}

\section*{Acknowledgement}
I would like to thank my supervisor, Professor Hiroyuki Ochiai, for his appropriate advice even under such circumstances. Without this advice, I believe that this paper was not completed.

\section{Preliminaries}

In this section, let $G$ be a finite group. First, we define notations on group actions to state main results of this paper. Main references of this section are \cite{MR} and \cite{Sch}.

\subsection{$G$-set}
First, we define $G$-sets. A $G$-set $X$ is defined as a set equipped with a map 
\[ \iota:G\times X\rightarrow X,\quad \iota(g,x)=gx\]
which satisfies
\[ g_1(g_2x)=(g_1g_2)x,\quad 1x=x\]
for any $g_1,g_2\in G$ and $x\in X$.

Note that the empty set $\emptyset$ has a $G$-set structure.

For any $G$-sets $S$ and $T$, the set $T^S$, which is the set of all maps from $S$ to $T$, is also has the $G$-set structure which is defined as
\begin{eqnarray}\label{GHom}
(gf)(s):=gf(g^{-1}s)
\end{eqnarray}
for any $g\in G$, $f:S\rightarrow T$ and $s\in S$.

For any $G$-sets $X$ and $Y$, the disjoint union and the cartesian product of $X$ and $Y$ are again $G$-sets.

Next, we define the notion of the isomorphism as $G$-sets. We call that two $G$-sets $X_1$ and $X_2$ are $G$-isomorphic if there exists a bijective map $f:X_1\rightarrow X_2$ such that $f(gx)=gf(x)$ holds for any $g\in G$ and $x\in X_1$. We call the map $f$ a $G$-isomorphism. 

For any $x\in X$, we denote an orbit of $x$ by $Gx$. A $G$-set $X$ is said to be transitive if $Gx=X$ for some $x\in X$. For example, the left quotient set $G/H$ is a transitive $G$-set for any subgroup $H\subset G$. Moreover, a transitive $G$-sets $X$ is $G$-isomorphic to $G/H$ for some subgroup $H\subset G$.

For transitive $G$-sets, the following proposition holds.
\begin{prop}[\text {\cite[p.111]{Knu}}]\label{Transit}
For any subgroups $H_1$ and $H_2$, two $G$-sets $G/H_1$ and $G/H_2$ are $G$-isomorphism if and only if there exists an element $g\in G$ such that $H_1=gH_2 g^{-1}$ holds. 
\end{prop}

Let $\Phi(G)$ be the set of conjugacy subgroups of $G$. By Proposition \ref{Transit}, a $G$-set $X$ is isomorphic to some disjoint union of $G$-sets $G/H$ where $H\in\Phi(G)$.

Let $H$ be a subgroup of $G$. We denote a $G$-set $X$ by $\Res{G}{H}(X)$ if we regard $X$ as an $H$-set. For the restriction map $\Res{G}{H}$, the following proposition holds.
\begin{prop}\label{Res}
Let $X$ and $Y$ be $G$-set. Then, two $H$-sets $\Res{G}{H}(X)$ and $\Res{G}{H}(Y)$ are $H$-isomorphic if $X$ and $Y$ are $G$-isomorphic. Moreover, two $H$-sets $\Res{G}{H}(X\cup Y)$ and $\Res{G}{H}(X)\cup\Res{G}{H}(Y)$ are $H$-isomorphic, and two $H$-sets $\Res{G}{H}(X\times Y)$ and $\Res{G}{H}(X)\times\Res{G}{H}(Y)$ are $H$-isomorphic
\end{prop}
\begin{proof}
If $X$ and $Y$ are $G$-isomorphic, then there exists $G$-isomorphism $\iota:X\rightarrow Y$. This map is also $H$-isomorphism, and hence, two $H$-sets $\Res{G}{H}(X)$ and $\Res{G}{H}(Y)$ are $H$-isomorphic.

Next, we consider the set $X\cup Y$ and $X\times Y$. The identity map on $X\cup Y$ and $X\times Y$ are $H$-isomorphic, then two $H$-sets $\Res{G}{H}(X\cup Y)$ and $\Res{G}{H}(X)\cup\Res{G}{H}(Y)$ are $H$-isomorphic and two $H$-sets $\Res{G}{H}(X)$ and $\Res{G}{H}(Y)$ are $H$-isomorphic.
\end{proof}

The following propositions will be used in \S 3.4 and \S 4.1.
%
\begin{prop}[\text {\cite[p.398]{Oh2}}]\label{TransM}
For any subgroups $K_1$ and $K_2$ of $G$, the cartesian product of two transitive $G$-sets $G/K_1$ and $G/K_2$ is $G$-isomorphic to
\[ \bigcup_{K_1 gK_2}G/K_1\cap gK_2 g^{-1} \]
where $g$ ranges over a set of double coset representatives of $K_1$ and $K_2$ in $G$.
\end{prop}

\begin{prop}[\text {\cite[p.400]{Oh2}}]\label{TransR}
For any subgroups $H$ and $K$ of $G$, the $H$-set $\Res{G}{H}(G/K)$ is $H$-isomorphic to
\[ \bigcup_{HgK}H/H\cap gK g^{-1} \]
where $g$ ranges over a set of double coset representatives of $H$ and $K$ in $G$.
\end{prop}

\begin{rmk}
Note that Proposition \ref{TransM} and \ref{TransR} were discussed when $G$ is not a finite group, however subgroups $H$ and $K$ satisfy that $|G/H|$ and $|G/K|$ is finite in \cite{Oh2}. For all subgroups $H$ of finite groups $G$ satisfies $|G/H|<\infty$.
\end{rmk}

\subsection{Calculation of $\mu_H(X)$}
We denote by $\mu_H(X)$ the number of orbits in the $G$-set $X$ which is $G$-isomorphic to $G/H$. We discuss a method of calculating $\mu_H(X)$ for any finite $G$-set $X$ and subgroup $H$ of $G$ with super characters, which was introduced by Knutson \cite{Knu}.

The set of all $G$-sets has an equivalent relation which is defined by $G$-isomorphisms. We denote an equivalent class of a finite $G$-set $X$ by $[X]$. 

We define $M'(G)$ by a set of all equivalence classes $[X]$. It has a semiring structure whose addition is defined via disjoint union and whose multiplication is defined via cartesian product. The zero element of $M'(G)$ is $[\emptyset]$, and the unit element $1$ of $M'(G)$ is the set whose cardinality is $1$.

We define the Burnside ring by the ring completion of $M'(G)$, written by $B(G)$. The Burnside ring $B(G)$ has a commutative ring structure (see Appendix A.6). The Burnside ring $B(G)$ has a $\mathbb{Z}$-basis $\{[G/H] \mid H\in\Phi(G) \}$.

For any $H\in \Phi(G)$, there exists an additive homomorphism $\mu_H:B(G)\rightarrow\mathbb{Z}$ such that $\mu_H([X])=\mu_H(X)$ holds for any finite $G$-set $X$. Then, we have
\[ [X]=\sum_{H\in\Phi(G)}\mu_{H}([X])[G/H] \]
for any finite $G$-set $X$. 

For any finite $G$-set $X$ and subgroup $H$ of $G$, we define a set $X_H$ by the set of $x\in X$ such that $hx=x$ holds for any $h\in H$. In addition, we define $\varphi_H(X)$ by the cardinality of $X_H$. Note that we have $\varphi_H(\emptyset)=0$ and $\varphi_H(1)=1$ for any subgroup $H$.

\begin{rmk}
Let $\chi$ be the permutation character associate with a finite $G$-set $X$. Then, we have $\varphi_{\circg{g}}(X)=\chi(g)$.
\end{rmk}

\begin{prop}\label{super}
The followings hold for any finite $G$-sets $S,T,X,Y$ and subgroups $H, H_1, H_2$.
\begin{itemize}
\item[{\rm(1)}] $\varphi_H(S)=\varphi_H(T)$ if $S$ and $T$ are $G$-isomorphism.
\item[{\rm(2)}] $\varphi_{H_1}(X)=\varphi_{H_2}(X)$ if $H_1$ and $H_2$ are conjugate.
\item[{\rm(3)}] $\varphi_H(X\cup Y)=\varphi_H(X)+\varphi_H(Y)$.
\item[{\rm(4)}] $\varphi_H(X\times Y)=\varphi_H(X)\varphi_H(Y)$.
\end{itemize}
\end{prop}
\begin{proof}
Identities (3) and (4) were proved in \cite[p.111]{Knu}. First, we prove the identity (1). 

Let $f:S\rightarrow T$ be a $G$-isomorphism. We show $f(S_H)=T_H$. For any $s\in S_H$, we have $hf(s)=f(hs)=f(s)$. Thus, we have $f(s)\in T_H$. Conversely, for any $t\in T_H$ we put $s=f^{-1}(t)$. Then $hs=hf^{-1}(t)=f^{-1}(ht)=f^{-1}(t)=s$ holds for any $h\in H$. Thus the element $s$ belongs to $S_H$. Hence, we have $f(S_H)=T_H$, that is, $\varphi_H(S)=\varphi_H(T)$ holds.

Next, we prove the identity (2). Take $g\in G$ which satisfies $H_1=g^{-1}H_2g$. Let $\iota:X\rightarrow X$ be the map defined by $\iota(x)=gx$ for any $x\in X$. The map $\iota$ is bijective. We show $\iota(X_{H_1})=X_{H_2}$. For any $x\in X_{H_1}$, we have $h_2\iota(x)=h_2gx=gh_1x=gx=\iota(x)$ for any $h_2\in H_2$ where $h_1\in H_1$ satisfies $gh_1=h_2g$. Thus, we have $\iota(x)\in X_{H_2}$. Conversely, for any $y\in X_{H_2}$, we put $x=\iota^{-1}(y)$. Then, $\iota(h_1x)=gh_1x=h_2gx=h_2\iota(x)=\iota(x)$ for any $h_1\in H_1$ where $h_2\in H_2$ satisfies $gh_1=h_2g$. Thus, we have $x\in X_{H_1}$, and hence, we have $\varphi_{H_1}(X)=\varphi_{H_2}(X)$.
\end{proof} 
Let $SCF(G)$ be the set of all maps from $\Phi(G)$ to $\mathbb{C}$. The set $SCF(G)$ has a commutative ring structure with the following operations.
\begin{eqnarray*}
(f_1+f_2)(H)&:=&f_1(H)+f_2(H),\\
(f_1 f_2)(H)&:=&f_1(H)f_2(H)
\end{eqnarray*}
where $f_1, f_2\in SCF(G)$ and $H\in\Phi(G)$. In \cite[p.110]{Knu}, an element in $SCF(G)$ is called a super central function on $G$. The zero element $0_{SFC(G)}$ is the map defined by $0_{SCF(G)}(H)=0$ for any $H\in\Phi(G)$ and the unit element of $1_{SFC(G)}$ is the map defined by $1_{SCF(G)}(H)=1$ for any $H\in\Phi(G)$.

By Proposition \ref{super}, there exists a ring homomorphism $\varphi:B(G)\rightarrow SCF(G)$ such that $\varphi([X])(H)=\varphi_H(X)$ holds for any finite $G$-set $X$. If $X$ is a finite $G$-set, then the map $\varphi([X])$ is called the super character of $G$-set $X$. In addition, the following theorem holds.

\begin{theo}\label{super1}

For any $H,V\in\Phi(G)$ satisfying $gHg^{-1}\subset V$ for some $g\in G$, there exists a unique rational integer $a_{H,V}$ such that 
\[ \mu_H(\alpha)=\sum_{H\lesssim V}a_{H,V}\varphi(\alpha)(V)\]
holds for any $\alpha\in B(G)$ where the notation $H\lesssim V$ means that there exists $g\in G$ such that $g^{-1}Hg\subset V$ holds. In particular, the map $\varphi$ is injective.
\end{theo}
Hence, if we know $\varphi_H(X)$ for all $H\in\Phi(G)$ and rational numbers $a_{H,V}$ for all $H,V\in\Phi(G)$, then we can calculate $\mu_H(X)$ for all $H\in\Phi(G)$.

To prove this theorem, we use the following lemma.
\begin{lemm}\label{super2}{\rm\cite[p.111]{Knu}}
Let $H_1$ and $H_2$ be subgroups of $G$. Then, one has $\varphi_{H_1}(G/H_2)\neq 0$ holds if and only if there exists an element $g\in G$ such that $g^{-1}H_1g\subset H_2$ holds.
\end{lemm}

\begin{proof}[Proof of Theorem \ref{super1}]
We prove this theorem by the induction on $|G/H|$. Recall that we have
\begin{eqnarray}\label{super111}
\varphi(\alpha)(H)=\sum_{V\in\Phi(G)}\mu_{V}(\alpha)\varphi_H(G/V)
\end{eqnarray}
for any $\alpha\in B(G)$.

First, we consider $H=G$. By (\ref{super111}) and Lemma \ref{super2}, we have $\mu_G(\alpha)=\varphi(\alpha)(G)=0$.

Next, we assume that this theorem holds for any $H\in\Phi(G)$ such that $|G/H|< n$ holds. Let $\alpha$ be an element of $B(G)$. By (\ref{super111}) and Lemma \ref{super2}, we have
\[ \varphi(\alpha)(H)=\sum_{V: H\lesssim V}\mu_{V}(\alpha)\varphi_H(G/V).\]
Thus,
\begin{eqnarray}\label{super112}
\mu_{H}(\alpha)=\dfrac{1}{\varphi_H(G/H)}\Big(\varphi(\alpha)(H)-\sum_{V: H\lesssim V, H\ncong V}\mu_{V}(\alpha)\varphi_H(G/V)\Big)
\end{eqnarray}
holds where $H\ncong V$ means that $H$ and $V$ are not conjugate. By (\ref{super112}) and the induction hypothesis, we can write 
\[\mu_H(\alpha)=\sum_{H\lesssim V}a_{H,V}\varphi(\alpha)(V)\]
for some rational numbers $a_{H,V}$'s.

Next, for any $V\in\Phi(G)$ satisfying $H\lesssim V$, let $a_{H,V}$ and $a'_{H,V}$ be rational numbers satisfying 
\begin{eqnarray}\label{super113}
\mu_{H}(\alpha)=\sum_{H\lesssim V}a_{H,V}\varphi(\alpha)(V)=\sum_{H\lesssim V}a'_{H,V}\varphi(\alpha)(V).
\end{eqnarray}
We show $a_{H,V}=a'_{H,V}$ by the induction of $|V/H|$. We substitute $\alpha=[G/H]$ in (\ref{super113}), then $1=a_{H,H}=a'_{H,H}$ holds. If $a_{H,V_1}=a'_{H,V_1}$ holds for any $V_1\in\Phi(G)$ such that $|V_1/H|<|V/H|$, then we substitute $\alpha=[G/V]$ in (\ref{super113}), thus we have
\[0=\sum_{H\lesssim V_1\lesssim V}a_{H,V_1}\varphi([G/V])(V_1)=\sum_{H\lesssim V_1\lesssim V}a'_{H,V_1}\varphi([G/V])(V_1).\]
Hence, we have $a_{H,V}=a'_{H,V}$ by the induction hypothesis. 

That is, this theorem holds.

\end{proof}

From the proof of Theorem \ref{super1}, we can calculate rational numbers $a_{H,V}$'s inductively. Moreover, there exists cases where calculation can be performed more effectively. For example, if $G$ is a cyclic group, we can use \Mobius\ inversion formula to calculate $a_{H,V}$. In the following example, we consider the case where $G$ is a cyclic group.

\begin{exa}\label{EXcr}
Let $n\geq 1$ be an integer and let $C_n$ be the cyclic group whose cardinality is $n$. We consider when $G=C_n$.

First, we investigate all elements of $\Phi(C_n)$. Note that the set of all subgroups of $C_n$ is $\{\circg{g^d}\mid d\ \mbox{divides}\ n\}$ where $g\in C_{n}$ is a generator of $C_n$. For any two distinct divisors $d_1$ and $d_2$ of $n$, we have $\circg{g^{d_1}}\neq \circg{g^{d_2}}$. Thus, for each divisor $d$ of $n$, there exists the subgroup uniquely whose cardinality is $n/d$. 

Moreover, the set of all subgroups of $C_n$ is $\Phi(C_n)$ since $C_n $ is an abelian group.

Now, we denote $\circg{g^{d}}$ by $C_{n/d}$. For any subgroups $C_{n/d}$ and $C_{n/d'}$ in $\Phi(C_n)$, we have
\begin{eqnarray}\label{Cn0}
 [\Res{C_n}{C_{n/d'}}(C_n /C_{n/d})]=(d,d')[C_{n/d'}/C_{n/[d,d']}].
\end{eqnarray}
In particular, we have
\begin{eqnarray}\label{Cn00}
 \varphi_{C_{n/d'}}(C_n/C_d)=\begin{cases}d' & \mbox{if}\ d'\mid d,\\ 0 & \mbox{if}\ d'\nmid d.\end{cases}
\end{eqnarray}
Hence, for any $C_n$-set $X$ and divisor $d$ of $r$, we have
\begin{eqnarray}\label{Cn1}
\varphi_{C_{n/d}}(X)=\sum_{d'\mid n}\mu_{C_{n/d'}}(X)\varphi_{C_{n/d}}(C_n/C_{n/d'})=\sum_{d'\mid d}d'\mu_{C_{n/d'}}(X).
\end{eqnarray}
By Theorem \ref{super1}, all $\mu_{C_{n/d}}(X)$'s are determined uniquely from all $\varphi_{C_{n/d}}(X)$'s. In fact, by the \Mobius\ inversion formula with the identity (\ref{Cn1}), we have
\begin{eqnarray}\label{Cn2}
 \mu_{C_{n/d}}(X)=\dfrac{1}{d}\sum_{d'\mid d}\mu\Big(\dfrac{d}{d'}\Big)\varphi_{C_{n/d'}}(X).
\end{eqnarray}
\end{exa}

Next, we consider the cartesian product of two finite $G$-sets $X$ and $Y$.

\begin{prop}\label{transtrans3}
Let $V_1, V_2, W_1, W_2, H_1,H_2$ be subgroups of $G$ and $V_1$, $V_2$ or $H_1$ is conjugate to $W_1$, $W_2$ or $H_2$, respectively. Then, we have $\mu_{H_1}(G/V_1\times G/V_2)=\mu_{H_2}(G/W_1\times G/W_2)$.
\end{prop}
\begin{proof}
Since $V_1$ is conjugate to $W_1$ and $V_2$ is conjugate $W_2$, the $G$-set $G/V_1\times G/V_2$ is $G$-isomorphic to $G/W_1\times G/W_2$ by Proposition \ref{Transit}. Moreover, since $H_1$ is conjugate to $H_2$, two maps $\mu_{H_1}$ is equal to $\mu_{H_2}$. Thus, we have $\mu_{H_1}(G/V_1\times G/V_2)=\mu_{H_2}(G/W_1\times G/W_2)$.
\end{proof}
\begin{defi}\label{transtransMD}
For any $V_1, V_2, H\in\Phi(G)$. we define the number $b_{V_1,V_2}(H)$ by $\mu_{H}(G/V_1\times G/V_2)$. By Proposition \ref{transtrans3}, this definition is well-defined.
\end{defi}
\begin{prop}\label{transtransM}
For any finite $G$-sets $X$ and $Y$ and $H\in\Phi(G)$, we have
\[ \mu_{H}(X\times Y)=\sum_{V_1,V_2\in\Phi(G)}b_{V_1,V_2}(H)\mu_{V_1}(X)\mu_{V_2}(Y)\]
\end{prop}
\begin{proof}
Two $G$-sets $X$ and $Y$ have the following form,
\[ [X]=\sum_{V_1\in\Phi(G)}\mu_{V_1}(X)[G/V_1],\quad [Y]=\sum_{V_2\in\Phi(G)}\mu_{V_2}(X)[G/V_2].\]
Thus, we have
\begin{eqnarray*}
[X\times Y]&=&\sum_{V_1,V_2\in\Phi(G)}\mu_{V_1}(X)\mu_{V_2}(Y)[G/V_1\times G/V_2]\\
&=&\sum_{H\in\Phi(G)}\Big(\sum_{V_1,V_2\in\Phi(G)}b_{V_1,V_2}(H)\mu_{V_1}(X)\mu_{V_2}(Y)\Big)[G/H].
\end{eqnarray*}
Hence, this proposition holds.
\end{proof}

In the remainder of \S 2.2, we fix a subgroup $H$ of $G$. By Proposition \ref{Res}, there are exist a ring homomorphism $\Res{G}{H}:R(G)\rightarrow R(H)$ such that $\Res{G}{H}([X])=[\Res{G}{H}(X)]$ holds for any finite $G$-set $X$.

\begin{prop}\label{transtrans5}
Let $K_1$ and $K_2$ be subgroups of $G$ and let $V_1, V_2$ be subgroup $H$. Suppose that $K_1$ and $K_2$ are conjugate as subgroups of $G$ and $V_1$ and $V_2$ are conjugate as subgroups of $H$. Then, we have $\mu_{V_1}(\Res{G}{H}(G/K_1))=\mu_{V_2}(\Res{G}{H}(G/K_2)$.
\end{prop}
\begin{proof}
Since $K_1$ and $K_2$ are conjugate, then $G/K_1$ and $G/K_2$ are $G$-isomorphic by Proposition \ref{Transit}, that is, two $H$-sets $\Res{G}{H}(G/K_1)$ and $\Res{G}{H}(G/K_2)$ are $H$-isomorphic. Moreover, since $V_1$ and $V_2$ are conjugate, then two maps $\mu_{H_1}$ is equal to $\mu_{H_2}$. Hence, this proposition holds.
\end{proof}

\begin{defi}\label{transtransRD}
Let $V$ be an element of $\Phi(G)$ and let $K$ be an element of $\Phi(H)$. We define the number $c_V(K)$ by $\mu_{K}(\Res{G}{H}(G/V))$. By Proposition \ref{transtrans5}, this definition is well-defined.
\end{defi}

\begin{prop}\label{transtransR}
For any finite $G$-set $X$ and $K\in\Phi(H)$, we have
\[ \mu_{K}(\Res{G}{H}(X))=\sum_{V\in\Phi(G)}c_{V}(K)\mu_{V}(X).\]
\end{prop}
\begin{proof}
A finite $G$-set $X$ has the following form,
\[ [X]=\sum_{V\in\Phi(G)}\mu_V(X)[G/V].\]
Then, we have
\begin{eqnarray*}
[\Res{G}{H}(X)]&=&\sum_{V\in\Phi(G)}\mu_V(X)[\Res{G}{H}(G/V)]\\
&=&\sum_{V\in\Phi(G)}\mu_V(X)\Big(\sum_{K\in\Phi(H)}c_V(K)[H/K]\Big)\\
&=&\sum_{K\in\Phi(H)}\Big(\sum_{V\in\Phi(G)}c_V(K)\mu_V(X)\Big)[H/K].
\end{eqnarray*}
Hence, this proposition holds.
\end{proof}

\subsection{Symmetric powers operations on $\lambda$-rings.}
In \S 2.3, we discuss symmetric powers operations on $\lambda$-rings and a character of symmetric powers of representations of finite groups which will be used in \S 3.2. We will obtain the generating function of the character of the $i$-th symmetric powers of a representation in \S 3.2.

For a relation between representations of finite groups and $\lambda$-rings, see \S2.2 and \S 2.3.

We show that calculating the generating function of the character of the $i$-th symmetric powers of a representation $\rho$ is equivalent to calculating $\lambda_t(\chi)$ where $\chi$ is the character of $\rho$.

\begin{defi}
Let $R$ be a $\lambda$-ring (which was defined in \S 2.2). We define operations $S^n:R\rightarrow R, n=0,1,2,\ldots$ by the following equation,
\[ S_t(r)=\dfrac{1}{\lambda_{-t}(r)}\]
for any $r\in R$, where $S_t(r):=\sum_{n=0}^{\infty}S^n(r)t^n$. The operation $S^n$ is said to be the $n$-th symmetric powers operation.
\end{defi}
Symmetric powers operations were defined in \cite{Mar}. For each finite group $G$, symmetric powers operations of $CF(G)$ give the character of symmetric powers representations of $G$ by the following proposition.
\begin{prop}
Let $\chi$ be the character of a representation $\rho$. Then, the character of the $n$-th symmetric powers of representation $\rho$ is $S^n(\chi)$ for any integer $n\geq 0$.
\end{prop}
\begin{proof}
Let $m$ be the degree of $\rho$. For each $g\in G$, let $\alpha_1,\ldots,\alpha_m$ be all eigenvalues of the linear map $\rho(g)$. It is known that 
\[ \lambda^n(\chi)(g)=\sum_{1\leq i_1<\cdots<i_n\leq m}\alpha_{i_1}\cdots\alpha_{i_n},\quad X(S^n\rho)
(g)=\sum_{1\leq i_1\leq\cdots\leq i_n\leq m}\alpha_{i_1}\cdots\alpha_{i_n},\]
where $X(S^n\rho)$ is the character of $S^n\rho$. Then, we have
\[ S_t(\chi)(g)=\dfrac{1}{\lambda_t(\chi)(g)}=\prod_{i=1}^m\dfrac{1}{1-\alpha_{i}t}=\sum_{n=0}^{\infty}X(S^n\rho)(g)t^n\]
where $S_t(\chi)(g):=\sum_{n=0}^{\infty}S^n(\chi)(g)$. Hence, $S^n(\chi)=X(S^n\rho)$ holds for any integer $n\geq 0$. 
\end{proof}

\section{$|A|$-colored $N$-nested $G$-set}
Let $G$ be a finite group. In this section, we define $|A|$-colored $N$-nested $G$-set, written by $J_{N,A,a}(X)$, for a given finite color set $A$, an element $a\in A$, a non-empty subset $N\subset \mathbb{N}\cup\{0\}$ and a finite $G$-set $X$.

In \S 3.1, we will define $|A|$-colored $N$-nested $G$-set and state its properties. In \S 3.2, we will put $|A|=2$ and $N=\mathbb{N}\cup \{0\}$, and discuss generating functions of the character of $n$-th symmetric power of representations of finite groups with $|A|$-colored $N$-nested $G$-set. Note that the discussion in \S 3.2 is also an idea of the definition and properties of $|A|$-colored $N$-nested $G$-set. In \S 3.3, we will put $N=\{0,1\}$ and establish the method of calculating the number of primitive colorings. In \S 3.4, we will generalize identities (\ref{necklaceM}) and (\ref{necklaceF}) for all $\mu_V(J_{N,A,a}(X))$ where $V\in\Phi(G)$.

In this section, we denote the set $\{1,\ldots,n\}$ by $[n]$ for any integer $n\geq 1$.

\subsection{Definition}
In \S 3.1, let $N$ be a non-empty subset of $\mathbb{N}\cup\{0\}$, let $A$ be a finite set with $|A|\geq 2$, and let $a$ be an element of $A$. First, we define $|A|$-colored $N$-nested $G$-set.
\begin{defi}
Let $X$ be a finite $G$-set. We define the $|A|$-colored $N$-nested $G$-set of $X$ with $a\in A$ by
\[ J_{N,A,a}(X):=\Big(\bigcup_{n\in N}{A'}^{[n]}\Big)^X, \]
where $A':=A\setminus\{a\}$ and $A'^{[0]}=\{a\}$. 

We regard the set $\Big(\bigcup_{n\in N}{A'}^{[n]}\Big)$ as the trivial $G$-set. We define a $G$-set structure on $J_{N,A,a}(X)$ by (\ref{GHom}). 
\end{defi}

Next, we state the decomposition of the disjoint union with the degree of $J_{N,A,a}(X)$.
\begin{defi}
For any elements of $a\in {A'}^{[n]}$, the symbol $\deg(a)$ is defined by $n$. We define $\deg(f)$ by
\[\deg(f):=\sum_{x\in X}\deg(f(x))\]
for any $f\in J_{N,A,a}(X)$. Next, we define the set $J_{N,A,a}^n(X)$ by
 \[ J_{N,A,a}^n(X):=\{ f\in J_{N,A,a}(X) \mid \deg(f)=n\} \]
for any integer $n\geq 0$.
\end{defi}
By this definition, we have 
\[ J_{N,A,a}(X)=\bigcup_{n=0}^{\infty}J_{N,A,a}^n(X).\]


\begin{prop}\label{JNProp}
Let $X$ be a finite $G$-set. For any integer $n\geq 0$, the set $J_{N,A,a}^n(X)$ is finite, and closed under the $G$-action of $J_{N,A,a}(X)$.
\end{prop}
\begin{proof}
For any integer $n\geq 0$, the set $A'^{[n]}$ is finite. Since an element $f\in J_{N,A,a}^n(X)$ satisfies $f(x)\in A^{[0]}\cup A^{[1]}\cup\cdots\cup A^{[n]}$ for any $x\in X$. Then, the set $J_{N,A,a}^n(X)$ is finite.

Next, for any $f\in J_{N,A,a}^n(X)$ and $g\in G$, we have
\[ \deg(gf)=\sum_{x\in X}\deg(gf(x))=\sum_{x\in X}\deg(f(g^{-1}x))=\deg(f)=n.\]
Hence, the element $gf$ belongs to $J_{N,A,a}^n(X)$.
\end{proof}

\begin{prop}
For given finite $G$-sets $X$ and $Y$ which are $G$-isomorphic, the $G$-set $J_{N,A,a}(X)$ is $G$-isomorphic to $J_{N,A,a}(Y)$.
\end{prop}
\begin{proof}
By the assumption, there exists a $G$-isomorphism $\iota: X\rightarrow Y$. We define a map $\iota_1:J_{N,A,a}(Y)\rightarrow J_{N,A,a}(X)$ by $\iota_1(f)(x):=f(\iota(x))$ for any $f\in J_{N,A,a}(Y)$ and $x\in X$. 
First, we prove that the map $\iota_1$ is injective. Let $f_{1}$ and $f_{2}$ be elements of $J_{N,A,a}(Y)$, and we assume $\iota_1(f_1)=\iota_1(f_2)$ holds. Then, we have $f_1(\iota(y))=f_2(\iota(y))$ holds for any $y\in Y$. Thus $f_1=f_2$ holds, that is, the map $\iota_1$ is injective.

Next we prove that the map $\iota_1$ is surjective. Let $f$ be an element of $J_{N,A,a}(X)$. We put $g\in J_{N,A,a}(X)$ which satisfies $g(x)=f(\iota^{-1}(x))$ for any $x\in X$. Thus, we have $\iota_1(g)(y)=g(\iota(y))=f(y)$, that is, the map $\iota_1$ is surjective.

Thus $\iota_1$ is bijective, and
\[ \iota_1(gf)(x)=gf(\iota(x))=f(g^{-1}\iota(x))=f(\iota(g^{-1}x))=\iota_1(f)(g^{-1}x)=g(\iota_1(f))(x)\]
holds for any $f\in J_{N,A,a}(Y)$, $g\in G$ and $x\in X$. 

Hence, the map $\iota_1$ is a $G$-isomorphism, that is, this proposition holds.
\end{proof}

\begin{lemm}\label{XYdisjointL}
Let $S,T$ and $U$ be $G$-sets. Then, the map $\iota_2: U^{S}\times U^{T}\rightarrow U^{S\cup T}$, which is defined by
\[ \iota_2(f,g)(u)=\begin{cases}f(u) & \mbox{if}\ u\in S, \\g(u) & \mbox{if}\ u\in T,\end{cases}\]
where $f\in U^{S}, g\in U^{T}$ and $u\in S\cup T$, is a $G$-isomorphism.
\end{lemm}
\begin{proof}
First, we prove that the map $\iota_2$ is injective. Let $f_{1}$ and $f_{2}$ be elements of $U^S$, and let $g_1$ and $g_2$ be elements of $U^T$. We assume that $\iota_2(f_1, g_1)=\iota_2(f_2,g_2)$ holds. If $u\in S$, then we have $f_1(u)=f_2(u)$. If $z\in T$, then we have $g_1(u)=g_2(u)$. Hence, $f_1=f_2$ and $g_1=g_2$ hold, that is, the map $\iota_2$ is injective.

Next we prove that the map $\iota_2$ is surjective. Let $f$ be an element of $U^{S\cup T}$. We put $f_1\in U^{S}$ and $g_1\in U^{T}$ which satisfies $f_1(x)=f(x)$ for any $x\in S$ and $g_1(x)=f(x)$ for any $y\in T$. Then, we have $\iota_2((f_1,g_1))=f$, that is, the map $\iota_2$ is surjective.

Finally, for any $g\in G$ and $(f_1,g_1)\in U^S\times U^T$, if $u\in S$ we have 
\[ g\iota_2((f_1,g_1))(u)=\iota_2((f_1,g_1))(g^{-1}u)=f_1(g^{-1}u)=(gf_1)(u)=\iota_2(g(f_1,g_1))(z).\]
 Similarly, we have $g\iota_2((f_1,g_1))(u)=\iota_2(g(f_1,g_1))(u)$ if $u\in T$. In any case, we have $g\iota_2((f_1,g_1))=\iota_2(g(f_1,g_2))$.

As a result, the map $\iota$ is a $G$-isomorphism.
\end{proof}

\begin{prop}\label{XYdisjointP}
Let $X$ and $Y$ be finite $G$-sets. Then, two $G$-sets $J_{N,A,a}(X \cup Y)$ and $J_{N,A,a}(X)\times J_{N,A,a}(Y)$ are $G$-isomorphic. In particular, for any integer $n\geq 1$ the following two $G$-sets are $G$-isomorphic:
\begin{eqnarray}\label{XYdisjoint}
\bigcup_{i+j=n}J_{N,A,a}^i(X)\times J_{N,A,a}^j(Y), \quad J_{N,A,a}^n(X\cup Y). 
\end{eqnarray}
\end{prop}

\begin{proof}
The first statement of this proposition holds by Lemma \ref{XYdisjointL}. Then, we show the second statement of this proposition.

Let $n\geq 1$ be an integer and let $\iota_3: J_{N,A,a}(X)\times J_{N,A,a}(Y)\rightarrow J_{N,A,a}(X\cup Y)$ be a map defined by Lemma \ref{XYdisjointL}. For any $f\in J^i_{N,A,a}(X)$ and $g\in J^j_{N,A,a}(Y)$ with $i+j=n$,
\[ \deg(\iota_3((f,g)))=\sum_{x\in X}\deg(f(x))+\sum_{y\in Y}\deg(g(y))=i+j=n\]
holds. Conversely, for any $h\in J^n_{N,A,a}(X\cup Y)$, $f\in J_{N,A,a}(X)$ and $g\in J_{N,A,a}(Y)$ with $\iota_3(f,g)=h$, we put $i=\deg(f)$ and $j=\deg(f)$. Then, $f\in J^i_{N,A,a}(X)$ and $g\in J^j_{N,A,a}(Y)$ hold. Hence,
\[ \iota_3\Big(\bigcup_{i+j=n}J_{N,A,a}^i(X)\times J_{N,A,a}^j(Y)\Big)= J_{N,A,a}^n(X\cup Y)\]
holds. That is, two $G$-sets (\ref{XYdisjoint}) are $G$-isomorphic. 
\end{proof}

Next, we consider the super character and orbits of $J^n_{N,A,a}(X)$.
\begin{defi}\label{generating}
Let $X$ be a finite $G$-set and let $H$ be a subgroup of $G$. By Proposition \ref{JNProp}, the set $J^n_{N,A,a}(X)$ is a finite $G$-set. We define two power serieses $\varphi_{H,t}(J_{N,A,a}(X))$ and $\mu_{H,t}(J_{N,A,a}(X))$ by
\[ \varphi_{H,t}(J_{N,A,a}(X)):=\sum_{n=0}^{\infty}\varphi_H (J_{N,A,a}^n(X) ) t^n \]
and
\[ \mu_{H,t}(J_{N,A,a}(X)):=\sum_{n=0}^{\infty}\mu_H(J_{N,A,a}^n(X))t^n \]
with an indeterminate variable $t$.
\end{defi}


By Proposition \ref{XYdisjointP}, we have
\begin{eqnarray}\label{XYdisjointP1}
\varphi_{H,t}(J_{N,A,a}(X\cup Y))=\varphi_{H,t}(J_{N,A,a}(X))\varphi_{H,t}(J_{N,A,a}(Y))
\end{eqnarray}
for any two finite $G$-sets $X$ and $Y$.

To calculate $\varphi_{H,t}(J_{N,A,a}(X))$, we show the following theorem.
\begin{theo}\label{Xpower}
For any subgroup $H$ of $G$ and finite $G$-set $X$, we have
\begin{eqnarray*}\label{Xpower0}
\varphi_{H,t}(J_{N,A,a}(X))=\prod_{i=1}^{\infty}\Big(\sum_{n\in N}(|A'|t^i)^n\Big)^{O_{X,H,i}}
\end{eqnarray*}
where $O_{X,H,i}$ is the number of orbits $Hx\subset \Res{G}{H}(X)$ such that $|Hx|=i$ holds for any integer $i\geq 1$.
\end{theo}

We prepare the following proposition.
\begin{prop}\label{XYRes}
Let $X$ be a finite $G$-set, and let $H$ be a subgroup of $G$. Then, two $G$-sets 
$\Res{G}{H}(J_{N,A,a}(X))$ and $J_{N,A,a}(\Res{G}{H}(X))$ are $H$-isomorphic.
\end{prop}
\begin{proof}
We show that the identity map on $J_{N,A,a}(X)$, written by $F$ in this proof, is a $G$-isomorphism. For any $f\in J_{N,A,a}(X)$, $x\in X$ and $h\in H$,
\[ hF(f)(x)=F(f)(h^{-1}x)=f(h^{-1}x)=hf(x)=F(hf)(x)\]
holds. Then, $\Res{G}{H}(J_{N,A,a}(X))$ and $J_{N,A,a}(\Res{G}{H}(X))$ are $H$-isomorphic.
\end{proof}

\begin{proof}[Proof of Theorem \ref{Xpower}]
First, we consider the case of $H=G$ and $X=G/K$. we prove
\begin{eqnarray}\label{Xpower1}
 \varphi_{G}(J_{N,A,a}^n(G/K))=\begin{cases} |A'|^{n/|G/K|} & \mbox{if}\ \frac{n}{|G/K|}\in N, \\ 0 & \mbox{otherwise} \end{cases}
\end{eqnarray}
for any subgroup $K$ of $G$. Any element $f \in J^n_{N,A,a}(G/K)$ satisfies $gf=f$ for all $g\in G$,
so that $f$ is a constant map since $G/K$ is a transitive $G$-set. Thus, there does not exist an element $f\in J^n_{N,A,a}(G/K)$ such that $|G/K| \mid n$ holds.

We assume $n\mid |G/K|$. Let $\alpha$ be an element of $G/K$ with $n'=\deg(f(\alpha))$. Then we have $n'\in N$ and $n'|G/K|=n$. The number of $f$ such that $f(\alpha)(m)\in A'$ holds for any $1\leq m\leq n'$ is $|A'|^{n'}$. Hence, we have (\ref{Xpower1}). 

Next, we show when $H=G$. By (\ref{Xpower1}), we have 
\[ \varphi_{G,t}(J_{N,A,a}(G/H))=\sum_{n\in N}(|A'|t^{|G/H|})^n.\]
Next, we prove this theorem for any finite $G$-set $X$. We use the orbit decomposition of $X$ and (\ref{XYdisjointP1}). Thus, we have
\begin{eqnarray}\label{XPower2}
 \varphi_{G,t}(J_{N,A,a}(X))=\prod_{i=1}^{\infty}\Big(\sum_{n\in N}(|A'|t^{i})^n\Big)^{O_{X,G,i}}.
\end{eqnarray}
Finally, we calculate the power series $\varphi_{H,t}(J_{N,A,a}(X))$ for any subgroup $H$ of $G$ with (\ref{XPower2}). By Proposition \ref{XYRes}, we have
\begin{eqnarray*}
\varphi_{H,t}(J_{N,A,a}(X))&=&\sum_{n=0}^{\infty}\varphi_H(J_{N,A,a}(X))t^n\\
&=&\sum_{n=0}^{\infty}\varphi_H(\Res{G}{H}(J_{N,A,a}(X)))t^n\\
&=&\sum_{n=0}^{\infty}\varphi_H(J_{N,A,a}(\Res{G}{H}(X)))t^n\\
&=&\varphi_{H,t}(J_{N,A,a}(\Res{G}{H}(X)))=\prod_{i=1}^{\infty}\Big(\sum_{n\in N}(|A'|t^i)^n\Big)^{O_{X,H,i}}.
\end{eqnarray*}
Hence, this theorem holds.
\end{proof}

To calculate $\mu_{H,t}(J_{N,A,a}(X))$, we use the following proposition.
\begin{prop}\label{superG1}
For any $V\in\Phi(G)$ and finite $G$-set $X$, 
\[ \mu_{H,t}(J_{N,A,a}(X))=\sum_{H\lesssim V}a_{H,V}\varphi_{V,t}(J_{N,A,a}(X))\]
holds where rational numbers $a_{H,V}$'s satisfy the identity of Theorem \ref{super1}.
\end{prop}
\begin{proof}
By Theorem \ref{super1}, we have
\begin{eqnarray*}
\mu_{H,t}(J_{N,A,a}(X))&=&\sum_{i=0}^{\infty}\mu_H(J^i_{N,A,a}(X))t^i\\
&=&\sum_{i=0}^{\infty}\sum_{H\lesssim V}a_{H,V}\varphi_{V}(J^i_{N,A,a}(X))t^i\\
&=&\sum_{H\lesssim V}a_{H,V}\varphi_{V,t}(J_{N,A,a}(X)).
\end{eqnarray*}
\end{proof}

\subsection{A relation with symmetric powers of representations of finite group}
We discuss a generating function of the character of the $n$-th symmetric powers of representation with \S 3.1. 

In \S 3.2, we denote $\mathbb{N}\cup\{0\}$ by $N_s$, and we assume $|A|=2$. We consider $J_{N_s,A,a}(X)$ for any finite $G$-set $X$.

Dress and Siebeneicher \cite[\S 2.13]{DS} defined the $n$-th symmetric power of a finite $G$-set $X$, written by $S^n(X)$, by the set of $f:X \rightarrow\mathbb{N}\cup\{0\}$ satisfying $\sum_{x\in X}f(x)=n$. The set $S^n(X)$ has a $G$-set structure by (\ref{GHom}). In this condition, an element $f\in S^n(X)$ can be written 
\[ f=x_1\cdots x_1\cdots x_m\cdots x_m \]
where $X=\{x_1,\dots,x_m\}$ and the element $x_i$ appears $f(x_i)$ times. Moreover, for any $g\in G$ we have $gf=gx_1\cdots gx_1\cdots gx_m\cdots gx_m$. Hence, the permutation representation associated with $S^n(X)$ is the $n$-th symmetric power of the permutation representation associated with $X$.

For $S^n(X)$, we have the following proposition.

\begin{prop}
The $G$-set $S^n(X)$ is $G$-isomorphic to $J^n_{N_s,A,a}(X)$ for any integer $n\geq 1$.
\end{prop}
\begin{proof}
From $|A|=2$, the set ${A'}^{[n]}$ has the unique element $f$ such that $f(k)=b$ holds for any $k=1,\ldots, n$, where $b$ is the unique element of $A\setminus\{a\}$. Thus, if $x,y\in \bigcup_{n\in N_s}^{\infty} {A'}^{[n]}$ satisfy $\deg(x)=\deg(y)$, then $x=y$ holds. 
 
We define a map $\iota_s:J^n_{N_s,A,a}(X)\rightarrow S^n(X)$ by $\iota_s(f)(x):=\deg(f(x))$ for any $x\in X$. 

First, we show that the map $\iota_s$ is injective. Let $f$ and $g$ be elements of $J^n_{N_s,A,a}(X)$ satisfying $\iota_s(f)=\iota_s(g)$. For any $x\in X$, we have $\deg(f(x))=\deg(g(x))$. Hence, we have $f(x)=g(x)$. That is, the map $\iota_s$ is injective.

Next, for any $g\in S^n(X)$, we put $f\in J^n_{N_s,A,a}(X)$ satisfying $f(x)\in A'^{[\deg(f(x))]}$. Then, we have $\iota_s(f)(x)=g(x)$. That is, the map $\iota_s$ is surjective.

Finally, we prove that $\iota_s(gf)=g\iota_s(f)$ holds for any $g\in G$ and $f\in J^n_{N_s,A,a}(X)$. One has 
\[ \iota_s(gf)(x)=\deg(gf(x))=\deg(f(g^{-1}x))=\iota_s(f(g^{-1}x)=g\iota_s(f)(x).\]
Then, the map $\iota_s$ is $G$-isomorphism.
\end{proof}

Substituting $N=N_s$ in the identities of Theorem \ref{Xpower}, we have
\begin{eqnarray}\label{Symrelation}
\varphi_{\circg{g},t}(J_{N_s,A,a}(X))=\prod_{i=1}^{\infty}\Big(\dfrac{1}{1-t^i}\Big)^{O_{X,\circg{g},i}}.
\end{eqnarray} 
For any $g\in G$, the right side of (\ref{Symrelation}) coincide with the generating function of the $S^n(\chi)(g)$ where $\chi$ is the character of the permutation representation associated with finite $G$-set $X$, and $S^n(\chi)$ is the character of the $n$-th symmetric power of permutation representation associated with a finite $G$-set $X$.

\begin{exa}\label{zeta}
Let $S_n$ be the symmetric group on $n$-letters. We consider the natural action of $S_n$ on $[n]$. Then, $S_n$-set $[n]$ is $S_n$-isomorphic to $S_n/S_{n-1}$ where we identify $S_{n-1}$ with the isotropy subgroup $\{\sigma \in S_n \mid \sigma (1)=1\}$ of $S_n$. 

Let $\{\lambda_1,\ldots,\lambda_m\}$ be the cycle structure of $\sigma \in S_n$. Then, the number of orbits of $\Res{S_n}{\circg{\sigma}}([n])$ whose cardinality is $i$, which is $O_{[n],\circg{\sigma},i}$, is equal to the number of $\lambda_k$ such that $\lambda_k=i$. Hence, we have 
\[ \varphi_{\circg{\sigma},t}(J_{N_s,A,a}([n]))=\prod_{i=1}^{\infty}\Big(\dfrac{1}{1-t^i}\Big)^{O_{[n],\circg{\sigma},i}}\]

\end{exa}
%
%
%
%

\subsection{Calculation of the number of primitive colorings}
In \S 3.3, we discuss $|A|$-colored $N$-nested $G$-set $X$ when $N=N_c:=\{0,1\}$ in order to calculate the number of primitive colorings. 

\begin{theo}\label{JPcolor}
For any finite $G$-set $X$, the map $\iota_c:J_{N_c,A,a}(X)\rightarrow A^X$ defined by 
\[ \iota_c(f)(x)=\begin{cases} a & \mbox{if}\ f(x)\in {A'}^{[0]}, \\ f(x)(1)& \mbox{if}\ f(x)\in {A'}^{[1]}, \end{cases} \]
is a $G$-isomorphism. In particular, two $G$-sets $J_{N_c,A,a}(X)$ and $A^X$ are $G$-isomorphic, and the set $J_{N_c,A,a}(X)$ is finite. Moreover, for any integer $n\geq 0$ and $f\in J^n_{N_c,A,a}(X)$, 
\[ |\{ x\in X \mid \iota_c(f)(x)\neq a \}|=n\]
 holds.
\end{theo}

\begin{proof}
First, we prove that the map $\iota_c$ is injective. Let $f$ and $g$ be elements of $J_{N_c,A,a}(X)$ with $\iota_c(f)=\iota_c(g)$. For any $x\in X$, if $f(x)\in {A'}^{[0]}$ then we have $f(x)=g(x)=a$. If $f(x)\in {A'}^{[1]}$, then $g(x)$ also belongs to $A'^{[1]}$ and $f(x)(1)=g(x)(1)$ holds. In any cases, we have $f(x)=g(x)$, thus, the map $\iota_c $ is injective.

Next, we prove that the map $\iota_c$ is surjective. For any $g\in A^X$, we put $f\in J_{N_c,A,a}(X)$ which satisfies 
\[ f(x)=\begin{cases} a & \mbox{if}\ g(x)=a,\\ 1\mapsto g(x) & \mbox{if}\ g(x)\neq a.\end{cases}\]
Thus, we have $\iota_c(f)(x)=g(x)$, that is, the map $\iota_c$ is surjective.

Finally, we prove $g\iota_c(f)=\iota_c(gf)$ for any $g\in G$ and $f\in J_{N_c,A,a}(X)$. Let $x$ be an element of $X$. If $f(g^{-1}x)\in {A'}^{[0]}$, then we have $g\iota_c(f)(x)=\iota_c(f)(g^{-1}x)=a$ and $\iota_c(gf)(x)=a$. If $f(g^{-1}x)\in {A'}^{[1]}$, then $(gf)(x)$ also belongs to ${A'}^{[1]}$, and $\iota_c(f)(g^{-1})(x)=f(g^{-1}x)(1)=(gf)(x)(1)=\iota_c(gf)(x)$ holds. In any case, we have $g\iota_c(f)=\iota_c(gf)$.
\end{proof}

We discuss a relation between the $G$-set $J_{N_c,A,a}(X)$ and the number of primitive colorings. First, we give the set of objects $X$ on polyhedrons, for example all vertices or all edges, and a rotation group $G$ of $X$. We assume that $X$ is a $G$-set.

\begin{defi}\label{color}
We define a coloring on $X$ with $A$ by an element of $A^X$. We define a primitive coloring on $X$ of $G$ with $A$ by a coloring on $X$ which is asymmetric under rotations of $G$. In other words, a primitive coloring is defined by an element $f\in A^X$ such that $|Gf|=|G|$ holds. 
\end{defi}
By the definition, the number of primitive colorings up to symmetry is $\mu_{C_1}(A^X)$.

An element $f\in \iota_c(J^n_{N_c,A,a}(X))$, which satisfies that the number of $x\in X$ such that $f(x)\neq a$ is equal to $n$, is a colorings that there exists $(|X|-n)$ times of $a$ colored place. The number of primitive colorings which has $(|X|-n)$ times of $a $ colored place is equal to $\mu_{C_1}(\iota_c(J^n_{N_c,A,a}(X)))$.

To calculate the number of primitive colorings, we calculate $\mu_{C_1}(J_{N_c,A,a}(X))$ and $\mu_{C_1}(J^{n}_{N_c,A,a}(X))$. Calculation of $\mu_{C_1}(J^n_{N_c,A,a}(X))$ is equivalent to calculating the generating function $\mu_{C_1,t}(J_{N_c,A,a}(X))$, and we obtain $\mu_{C_1,t}(J_{N_c,A,a}(X))$ by substituting $H=C_1$ in the identity of Proposition \ref{superG1}. 

Thus, we investigate rational numbers $a_{V,W}$'s and $\varphi_{H,t}(J^{n}_{N_c,A,a}(X))$ for any subgroup $H$ of $G$. To calculate $\varphi_{H,t}(J^{n}_{N_c,A,a}(X))$, we use the following theorem.

\begin{theo}\label{73eq}
For any finite $G$-set $X$ and subgroup $H$ of $G$,
\begin{eqnarray}\label{73eq1}
\varphi_{H,t}(J_{N_c,A,a}(X))=\prod_{i=1}^{\infty}(1+|A'|t^i)^{O_{X,H,i}}
\end{eqnarray}
holds. In particular,
\begin{eqnarray}\label{73eq2}
\varphi_H(J_{N_c,A,a}(X))=|A|^{O_H(X)}
\end{eqnarray}
holds where $O_H(X)$ is the number of orbits of $\Res{G}{H}(X)$ as an $H$-set.
\end{theo}

To prove this theorem, we state the following lemmas.

\begin{lemm}\label{73eqlemm1}
Let $X$ be a finite $G$-set. Then, an element $f\in J_{N_c,A,a}(X)$ satisfies $\deg(f)\leq |X|$. In particular, the set $J^n_{N_c,A,a}(X)$ is the empty set for any integer $n>|X|$.
\end{lemm}
\begin{proof}
From $N_c=\{0,1\}$, the degree of $f(x)$ is $0$ or $1$ for any $x\in X$. Then, 
\[ \deg(f)=\sum_{x\in X}\deg(f(x))\leq |X|\]
holds. In particular, there does not exists an element $f\in J_{N_c,A,a}(X)$ such that $\deg(f)>|X|$. Hence, the set $J^n_{N_c,A,a}(X)$ is the empty set for any integer $n>|X|$.
\end{proof}

The following lemma shows that $\varphi_{H,t}$ is a refinement of $\varphi_H$ and that $\mu_{H,t}$ is that of $\mu_H$.

\begin{lemm}\label{73eqlemm2}
Let $X$ be a finite $G$-set and let $H$ be a subgroup of $G$. Consider that we substitute $t=1$ in $\mu_{H,t}(J_{N_c,A,a}(X))$ and $\varphi_{H,t}(J_{N_c,A,a}(X))$. Then,
\[ \mu_{H,1}(J_{N_c,A,a}(X))=\mu_{H}(J_{N_c,A,a}(X)),\quad \varphi_{H,1}(J_{N_c,A,a}(X))=\varphi_{H}(J_{N_c,A,a}(X)) \]
hold.
\end{lemm}
\begin{proof}
By Definition \ref{generating}, we recall
\[\mu_{H,t}(J_{N_c,A,a}(X))=\sum_{n=0}^{\infty}\mu_{H}(J^n_{N_c,A,a}(X))t^n.\]
By Lemma \ref{73eqlemm1}, we have $\mu_{H}(J^n_{N_c,A,a}(X))=0$ for any integer $n> |X|$. Substituting $t=1$ in $\mu_{H,t}(J_{N_c,A,a}(X))$, we have
\begin{eqnarray*}
\mu_{H,1}(J_{N_c,A,a}(X))&=&\sum_{n=0}^{\infty}\mu_{H}(J^n_{N_c,A,a}(X))\\
&=&\sum_{n=0}^{|X|}\mu_{H}(J^n_{N_c,A,a}(X))\\
&=&\mu_{H}\Big(\bigcup_{n=0}^{|X|}J^n_{N_c,A,a}(X)\Big)=\mu_H(J_{N_c,A,a}(X)).
\end{eqnarray*}
Similarly, we have $\varphi_{H,1}(J_{N_c,A,a}(X))=\varphi_{H}(J_{N_c,A,a}(X))$ since Proposition \ref{super} (3) and 
\[\varphi_{H,t}(J_{N_c,A,a}(X))=\sum_{n=0}^{\infty}\mu_{H}(J^n_{N_c,A,a}(X))t^n\] 
hold.
\end{proof}

\begin{proof}[Proof of Theorem \ref{73eq}]
For (\ref{73eq1}), we substitute $N=N_c$ in the identities of Theorem \ref{Xpower}. Thus, we have (\ref{73eq1}).

For (\ref{73eq2}), we substitute $t=1$ in the identity (\ref{73eq1}). The left side of (\ref{73eq1}) is equal to $\varphi_H(J_{N_c,A,a}(X))$ since Lemma \ref{73eqlemm2} holds and the right side of (\ref{73eq1}) is 
\[ |A|^{O_{X,H,1}+O_{X,H,2}+\cdots+O_{X,H,|X|}}=|A|^{O_H(X)}\]
by the definition of $O_H(X)$ and $O_{X,H,i}$'s.
\end{proof}

\begin{exa}\label{zeta2}
Let $n\geq 1$ be an integer. We denote the set of all vertices of a regular $n$-gon by $X_n$, and we denote by $C_n$ the cyclic group of order $n$, which is the group of rotation symmetry of 
the regular $n$-gon. Then, we can regard $X_n$ as a $C_n$-set, and the set $X_n$ is isomorphic to $C_n/C_1$. 

We show $\mu_{C_1}(J_{N_c,A,a}(X))=M(|A|,n)$ which is the same result as the number of primitive necklaces \cite{MR}. First, we calculate $\varphi_{C_{n/d},t}(J_{N_c,A,a}(X_n))$ for any divisor $d$ of $n$. By (\ref{Cn0}), we have
\[ O_{X_n,C_{n/d},i}=\begin{cases} d & \mbox{if}\ i=\frac{n}{d}, \\ 0 & \mbox{otherwise}. \end{cases} \]
By (\ref{73eq1}), we have
\[ \varphi_{C_{n/d},t}(J_{N_c,A,a}(X_n))=(1+|A'|t^{n/d})^d.\]
Thus, for any divisor $d$ of $n$, we have 
\begin{eqnarray}\label{EXCRR}
\begin{split}
\mu_{C_{n/d},t}(J_{N_c,A,a}(X_n))=&\sum_{i=0}^{\infty}\mu_{C_{n/d}}(J^i_{N_c,A,a}(X_n))t^i\\
=&\dfrac{1}{d}\sum_{i=0}^{\infty}\sum_{d'\mid d}\mu\Big(\dfrac{d}{d'}\Big)\varphi_{C_{n/d'}}(J^i_{N_c,A,a}(X_n))t^i\\
=&\dfrac{1}{d}\sum_{d'\mid d}\mu\Big(\dfrac{d}{d'}\Big)\varphi_{C_{n/d'},t}(J_{N_c,A,a}(X_n))\\
=&\dfrac{1}{d}\sum_{d'\mid d}\mu\Big(\dfrac{d}{d'}\Big)(1+|A'|t^{n/d'})^{d'}
\end{split}
\end{eqnarray}
where the third equation follows from (\ref{Cn2}) in the case of $J^i_{N_c,A,a}(X)$. Substituting $t=1$, we have
\begin{eqnarray}\label{ngonnecklace}
\mu_{C_{n/d}}(J_{N_c,A,a}(X_n))=\dfrac{1}{d} \sum_{d'\mid d}\mu\Big(\dfrac{d}{d'}\Big)|A|^{d'}=M(|A|,d)
\end{eqnarray}
by Lemma \ref{73eqlemm2} and (\ref{necklace}). We substitute $d=n$ in (\ref{ngonnecklace}), and we obtain $\mu_{C_1}(J_{N_c,A,a}(X_n))=M(|A|,n)$.

For example, we consider when $n=6$, $A=\{0,1\}$ and $a=2$. Substituting these condition and $d=n=6$ in (\ref{EXCRR}). Then we have
\begin{eqnarray*}
\mu_{C_1,t}(J_{N_c,A,a}(X_6))&=&\dfrac{1}{6}\Big((1+t)^6-(1+t^2)^3-(1+t^3)^2+(1+t^6)\Big)\\
&=&t+2t^2+3t^3+2t^4+t^5.
\end{eqnarray*}
Substituting $t=1$, we have $\mu_{C_1}(J_{N_c,A,a}(X_6))=9$ by Lemma \ref{73eqlemm2}. Figure 3 is all primitive colorings on $6$-gon of $C_6$ with $A=\{0,1\}$.
\begin{figure}[h]
	\centering
	\includegraphics[width=8cm,height=8cm]{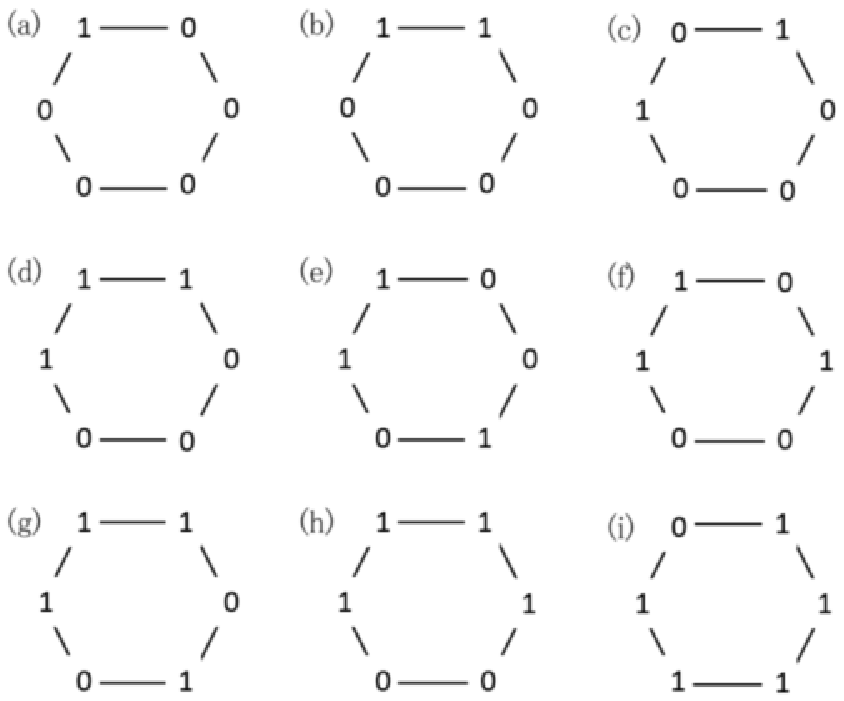}\\
	Figure 3: Primitive colorings on vertices of a regular 6-gon with $2$ colors $\{0,1\}$.
\end{figure}

Recall that the coefficient of $t^n$ in $\mu_{C_1,t}(J_{N_c,A,a}(X_6))$ is the number of primitive colorings on $X_6$ of $C_6$ which has $(6-n)$ times of $0$-colored places. In Figure 3, the image (a) is the unique primitive coloring which has five $0$-colored places, images (b) and (c), or (g) and (h) are primitive colorings which have four or two $0$-colored places, respectively, images (c), (d) and (e) are primitive colorings which have three $0$-colored places, the image (i) is the unique primitive coloring which has one $0$-colored place.
\end{exa}

\subsection{Generalization of identities of necklace polynomials}

In \S 3.4, we generalize identities (\ref{necklaceM}) and (\ref{necklaceF}) for $\mu_V(J_{N_c,A,a}(X))$ where $V$ is an element of $\Phi(G)$, $A$ is a finite set with $|A|\geq 2$ and $a\in A$ and $X$ is a finite $G$-set $X$. We show that if $X=X_n$, then we obtain identities (\ref{necklaceM}) and (\ref{necklaceF}).

\begin{theo}\label{GeneM}
Let $X$ be a finite $G$-set. Take finite sets $A$ and $B$ with $|A|, |B|\geq 2$, and take elements $a\in A$ and $b\in B$. Then, two $G$-sets \\$J_{A\times B,N_c,(a,b)}(X)$ and $J_{A,N_c,a}(X) \times J_{B,N_c,b}(X)$ are $G$-isomorphic. In particular, one has
\begin{eqnarray}\label{Power2}
\begin{split}
&\mu_{H}(J_{N_c,A\times B,(a,b)}(X))\\
=&\sum_{V_1, V_2\in\Phi(G)}b_{V_1,V_2}(H)\mu_{V_1}(J_{N_c,A,a}(X))\mu_{V_2}(J_{N_c,B,b}(X))
\end{split}
\end{eqnarray}
for any $H\in\Phi(G)$.
\end{theo}
For the number $b_{V_1,V_2}(H)$, recall Definition \ref{transtransMD}.
\begin{proof}
We calculate both super characters of two $G$-sets $J_{N_c,A\times B, (a,b)}(X)$ and $J_{N_c,A,a}(X)\times J_{N_c,B,b}(X)$. By (\ref{73eq2}), we have
\begin{eqnarray*}
\varphi_H(J_{N_c,A\times B,(a,b)}(X))&=&(|A||B|)^{O_H(X)}\\
&=&|A|^{O_H(X)}|B|^{O_H(X)}=\varphi_H(J_{N_c,A,a}(X)\times J_{N_c,B,b}(X)).
\end{eqnarray*}
for any subgroup $H\subset G$.

For (\ref{Power2}), we calculate two integers $\mu_H(J_{A\times B,N_c,(a,b)}(X))$ and \\$\mu_H(J_{A,N_c,a}(X) \times J_{B,N_c,b}(X))$. Hence, we obtain the identity (\ref{Power2}) by Proposition \ref{transtransM}.
\end{proof}

\begin{exa}
Under the assumption of Example \ref{zeta2}, we show that the identity (\ref{GeneM}) is a generalization of the identity (\ref{necklaceM}).

Substituting $X=X_n$ in (\ref{Power2}), we have
\begin{eqnarray}\label{Power22}
\begin{split}
&\mu_{C_1}(J_{N_c,A\times B,(a,b)}(X_n))\\
=&\sum_{i,j\mid n}b_{C_{n/i}, C_{n/j}}(C_n)\mu_{C_{n/i}}(J_{N_c,A,a}(X_n))\mu_{C_{n/j}}(J_{N_c,B,b}(X_n)).
\end{split}
\end{eqnarray}

We use Proposition \ref{TransM} to calculate numbers $b_{C_{n/i}, C_{n/j}}$. For any divisors $i$ and $j$ of $n$, if $[i,j]=n$ then the number of double cosets $C_{n/i}gC_{n/j}$ in $C_n$ such that $C_{n/i}\cap C_{n/j}=C_n$ is $(i,j)$. If $[i,j]\neq n$, there does not exist a double coset satisfying $C_{n/i}\cap C_{n/j}=C_n$. With this fact and (\ref{ngonnecklace}), we have
\[ M(k_1k_2,n)=\sum_{[i,j]=n}(i,j)M(k_1,i)M(k_2,j)\]
where $k_1=|A|$ and $k_2=|B|$. It is the same as (\ref{necklaceM}).
\end{exa}

Next, we consider a generalization of (\ref{necklaceF})
\begin{theo}\label{GeneF}
Let $X$ be a finite $G$-set and let $r\geq 1$ be an integer. Take a finite set $|A|$ with $|A|\geq 2$, and an element $a\in A$. We put $P=A\times\cdots\times A$ $($$r$-times of cartesian product$)$ and $p=(a,\ldots,a)\in P$. 

Suppose that there exists a finite group $G'$ and a $G'$-set $Y$ such that $G$ is embedded to $G'$ and $[\Res{G'}{G}(Y)]=r[X]$ holds. Then the $G$-set $J_{N_c,P,p}(X)$ is $G$-isomorphic to $\Res{G'}{G}(J_{N_c,A,a}(Y))$. In particular, one has
\begin{eqnarray}\label{GeneFF}
\mu_{H}(J_{N_c,P,p}(X))=\sum_{V\in \Phi(G')}c_{V}(H)\mu_{V}(J_{N_c,A,a}(Y))
\end{eqnarray}
for any $H\in\Phi(G)$.
\end{theo}
For the number $c_V(H)$, recall Definition \ref{transtransRD}.
\begin{proof}
We calculate both super characters of two $G$-sets $J_{N_c,P,p}(X)$ and \\ $\Res{G'}{G}(J_{N_c,A,a}(Y))$.  For any subgroup $H$ of $G$, we have
\begin{eqnarray*}
\varphi_H(J_{N_c,P,p}(X))&=&|P|^{O_H(X)}\\
&=&|A|^{rO_H(X)}\\
&=&|A|^{O_H(\Res{G'}{G}(Y))}\\
&=&\varphi_H(J_{N_c,A,a}(\Res{G'}{G}(Y)))=\varphi_H(\Res{G'}{G}(J_{N_c,A,a}(Y))).
\end{eqnarray*}
where the first and fourth equality follows from (\ref{73eq2}) and fifth equality follows from Proposition \ref{XYRes}.

Calculating $\mu_H(J_{N_c,P,p}(X))$ and $\mu_H(\Res{G'}{G}(J_{N_c,A,a}(Y)))$ we obtain the identity (\ref{GeneFF}) by Proposition \ref{transtransR}.
\end{proof}

\begin{exa}
Under the assumption of Example \ref{zeta2}, we show that the identity (\ref{GeneF}) is a generalization of the identity (\ref{necklaceF}). Let $r\geq 1$ be an integer.

We can identify $C_n$ with $C_{nr/r}\subset C_{nr}$, and we have $[\Res{C_{nr}}{C_{nr/r}}(X_{nr})]=r[X_n]$.


Then, two $C_n$-sets $J_{N_c,P,p}(X_n)$ and $ \Res{C_{nr}}{C_{nr/r}}(J_{N_c,A,a}(X_{nr}))$ are $C_n$-isomorphic. In particular, we have 
\begin{eqnarray*}
\mu_{C_1}(J_{N_c,P,p}(X_n))=\sum_{d\mid nr}c_{C_{nr/d}}(C_1)\mu_{C_{nr/d}}(J_{N_c,A,a}(X_{nr}))
\end{eqnarray*}
We use Proposition \ref{TransR} to calculate numbers $c_{C_{nr/d}}(C_1)$. The number of double cosets $C_{nr/d}gC_{nr/r}$ in $C_{nr}$ such that $C_{nr/r}\cap C_{nr/d}=C_1$ is $(d,r)$ when $[d,r]=nr$. In this case, we have $(d,r)=d/n$. Ir $[d,r]\neq nr$, then there does not exists a double coset $C_{nr/d}gC_{nr/r}$ in $C_{nr}$ such that $C_{nr/r}\cap C_{nr/d}=C_1$. Thus, we have
\[ M(k^r,n)=\sum_{[d,r]=nr}\dfrac{d}{n}M(k,d)\]
which is the same as (\ref{necklaceF}).
\end{exa}

We assume that the subgroup $C_1 \times G\subset C_r\times G$ is identified with $G$. Substituting $Y=C_r\times X$ in Theorem \ref{GeneF}, we have the following corollary.
\begin{cor}
For any finite $G$-set $X$, two $G$-sets $\Res{C_r\times G}{G}(J_{N_c,A,a}(C_r\times X))$ and $J_{N_c,P,p}(X)$ are $G$-isomorphic. In particular, one has
\begin{eqnarray*}
\mu_{H}(J_{N_c,P,p}(X))=\sum_{V\in\Phi(C_r\times G)}c_{V}(H)\mu_{V}(J_{N_c,A,a}(C_r\times X))
\end{eqnarray*}
for any subgroup $H$.
\end{cor}

\section{The number of primitive colorings on some objects of polyhedrons.}
In this section, we assume $N=N_c=\{0,1\}$ and calculate the number of primitive colorings on a given the set of object $X$ and a rotation group $G$ of $X$ with a color set $A$. For more detail, we calculate $\mu_{C_1}(J_{N_c,A,a}(X))$ and \\ $\mu_{C_1,t}(J_{N_c,A,a}(X))$ in the following cases as examples where $a\in A$.
\begin{itemize}
\item[(i)] We assume that $X$ is the set of all vertices of $n$-prism and $G$ is the dihedral group $D_{n}$.
\item[(ii)] We assume that $X$ is the set of all vetrices of $n$-gon and $G$ is the dihedral group.
\end{itemize}
In Example \ref{zeta2}, we considered the set of all vertices of $n$-gon $X_n$ and the cyclic group $C_n$. The difference of the condition of Example \ref{zeta2} and (ii) is that we consider not only rotations but also reverses in (ii).
In \S 4.1.1, we investigate the cardinality of orbits of $\Res{D_n}{K}(D_n/K) $ and $\varphi_K(D_{n}/K')$ for any $K, K'\in \Phi(D_{n})$. In addition, we represent $\mu_K(X)$ as all $\varphi_K(X)$ for any $K\in\Phi(D_{n})$ (Corollary \ref{DDDDD2}). In \S 4.1.2, we consider the case (ii), and in \S 4.1.3 we consider the case (ii).

\subsection{A dihedral group }
We consider the case of $G=D_{n}$ which is a finite group generated by two elements $a$ and $b$ satisfying $a^n=b^2=1$ and $bab=a^{-1}$. In \S 4.1, we redefine $C_{n/d}$ by the subgroup $\circg{a^d}$ for any divisor $d$ of $n$.

\subsubsection{Calculation of orbits}
First, we investigate all elements of $D_{n}$.
\begin{prop}\label{CongDn}
Conjugacy subgroups of $D_{n}$ are the following subgroups.
\begin{eqnarray*}
C_{n/d}&:=&\{1, a^d,\ldots, a^{(n/d-1)d} \},\\
D_{n/d}&:=&C_{n/d}\cup bC_{n/d}, \\
D'_{n/d}&:=&C_{n/d}\cup baC_{n/d},
\end{eqnarray*}
for each divisor $d$ of $n$ where two subgroups $D_{n/d}$ and $D'_{n/d}$ are conjugate if $2\nmid d$.
\end{prop}
\begin{proof}
Put $D_{n}=\{1,a,a^2,\ldots,a^{n-1}, b,ba,ba^2,\ldots,ba^{n-1}\}$. First, we investigate all subgroups $H$ of $D_{n}$. Let $d$ be the minimum number such that $a^d$ is an element of $H$. Thus, $C_{n/d}$ is contained to $H$. If there does not exist an element which has the form $ba^k$ in $H$, then $H=C_{n/d}$ holds. If $H$ has an element which has the form $ba^k$, put $k=0,\ldots,n-1$ which is the minimum number such that $ba^k$ is an element of $H$. Hence, we have $H=C_{n/d}\cup ba^kC_{n/d}$. 

Next, we consider all conjugacy subgroups of $D_{n}$. Let $d$ be a divisor of $n$. For elements $a,b\in D_{n}$, we have $aC_{n/d}a^{-1}=C_{n/d}$ and $bC_{n/d}b^{-1}=C_{n/d}$ by the definition of $D_{n}$. Thus, $C_{n/d}$ is one of conjugacy subgroup of $D_{n}$. In addition, we have $a(C_{n/d}\cap ba^k C_{n/d})a^{-1}=C_{n/d}\cap ba^{k-2}$ and $b(C_{n/d}\cap ba^k C_{n/d})b^{-1}=C_{n/d}\cap ba^{-k}$. Now, put
\[D_{n/d}:=C_{n/d}\cup bC_{n/d},\quad D'_{n/d}:=C_{n/d}\cup baC_{n/d}.\]
Thus, all subgroups of $D_{n}$ which has the form $C_{n/d}\cup ba^{k}C_{n/d}$ are conjugate to either of the above subgroups. Moreover, if $2\nmid d$ then the above two subgroups are conjugate. If $2\mid d$ then the above two subgroups are not conjugate.
\end{proof}

We investigate orbits of $K$-set $\Res{D_{n}}{K}(D_{n}/K')$ for any $K, K'\in\Phi(D_{n})$. For the calculation, we recall Proposition \ref{TransR}, and we focus on the cardinality of orbits of $\Res{D_{n}}{K}(D_{n}/K')$ from (\ref{73eq1}).
\begin{itemize}
\item [(i)] First, we consider the case of $K'=C_{n/d'}$. Recall that $C_{n/d'}$ is the normal subgroup of $D_{n}$. Then, we have
\begin{eqnarray*}
[\Res{D_{n}}{C_{n/d'}}(D_{n}/K)]=\sum_{C_{n/d'}gK}[ C_{n/d'}/(C_{n/d'}\cap gKg^{-1})]
\end{eqnarray*}
where $g$ ranges over a set of double coset representatives of $C_{n/d'}$ and $K$ in $D_{n}$. For any $g\in D_{n}$, we have $|(C_{n/d'}\cap gKg^{-1})|=|C_{n/d'}\cap K|$. In addition, the number of double cosets of $C_{n/d'}$ and $K$ is $|D_{n}/C_{n/d'}K|$. Thus, we have
\begin{eqnarray}\label{DD3CC}
O_{D_{n}/K, C_{n/d'},i}=\begin{cases} |D_n/C_{n/d'}K| & \mbox{if}\ i=|C_{n/d'}/(C_{n/d'}\cap K)|, \\ 0 & \mbox{otherwise}. \end{cases}
\end{eqnarray}
Hence, we have
\begin{eqnarray*}
\varphi_{C_{n/d'}}(D_{n}/K)=\begin{cases} |D_{n}/C_{n/d'}K| & \mbox{if}\ C_{n/d'}\subset K, \\ 0 & \mbox{otherwise}, \end{cases}
\end{eqnarray*}
and by (\ref{73eq1}) and (\ref{DD3CC}), we have
\begin{eqnarray*}
\varphi_{C_{n/d'},t}(J_{N_c,A,a}(D_{n}/K))=(1+|A'|t^{|C_{n/d'}/(C_{n/d'}\cap K)|})^{|D_{n}/C_{n/d'}K|}.
\end{eqnarray*}
In particular,
\begin{eqnarray}
\label{DD1CC}\varphi_{C_{n/d'}}(D_{n}/C_{n/d})&=&\begin{cases} 2d' & \mbox{if}\ d'\mid d, \\ 0 & \mbox{if}\ d'\nmid d, \end{cases}\\
\label{DD1CD}\varphi_{C_{n/d'}}(D_{n}/D_{n/d})&=&\begin{cases} d' & \mbox{if}\ d'\mid d, \\ 0 & \mbox{if}\ d'\nmid d, \end{cases}\\
\label{DD1CE}\varphi_{C_{n/d'}}(D_{n}/D'_{n/d})&=&\begin{cases} d' & \mbox{if}\ d'\mid d, \\ 0 & \mbox{if}\ d'\nmid d, \end{cases}
\end{eqnarray}
and 
\begin{eqnarray}
\label{DD2CC}\varphi_{C_{n/d'},t}(J_{N_c,A,a}(D_{n}/C_{n/d}))&=&(1+|A'|t^{[d,d']/d'})^{2(d,d')},\\
\label{DD2CD}\varphi_{C_{n/d'},t}(J_{N_c,A,a}(D_{n}/D_{n/d}))&=&(1+|A'|t^{[d,d']/d'})^{(d,d')},\\
\label{DD2CE}\varphi_{C_{n/d'},t}(J_{N_c,A,a}(D_{n}/D'_{n/d}))&=&(1+|A'|t^{[d,d']/d'})^{(d,d')}
\end{eqnarray}
hold. Identities (\ref{DD2CD}) and (\ref{DD2CE}) are equal.

\item [(ii)] We consider the case of $K=C_{n/d}$. We discuss the method similar to (i). We have
\begin{eqnarray*}
\Res{D_{n}}{K}(D_{n}/C_{n/d})]=\sum_{KgC_{n/d}}[ K/K\cap gC_{n/d}g^{-1}]
\end{eqnarray*}
where $g$ ranges over a set of double coset representatives of $K$ and $C_{n/d}$ in $D_{n}$. For any $g\in D_n$, we have $|K\cap gC_{n/d}g^{-1})|=|K\cap C_{n/d}|$. In addition, the number of double cosets of $C_{n/d}$ and $K$ is $|D_{n}/C_{n/d}K|$. Thus, we have
\begin{eqnarray}\label{DD4CC}
O_{D_{n}/C_{n/d}, K,i}=\begin{cases} |D_n/C_{n/d}K| & \mbox{if}\ i=|K/(K\cap C_{n/d})|, \\ 0 & \mbox{otherwise}. \end{cases}
\end{eqnarray}
Hence, we have
\begin{eqnarray*}
\varphi_{K}(D_{n}/C_{n/d})=\begin{cases} |D_{n}/C_{n/d}K| & \mbox{if}\ K\subset C_{n/d}, \\ 0 & \mbox{otherwise}, \end{cases}
\end{eqnarray*}
and by (\ref{73eq1}) and (\ref{DD4CC}), we have
\begin{eqnarray*}
\label{DD1}\varphi_{K,t}(J_{N_c,A,a}(D_{n}/C_{n/d}))=(1+|A'|t^{|K/K\cap C_{n/d'})|})^{|D_{n}/C_{n/d}K|}.
\end{eqnarray*}
In particular,
\begin{eqnarray}
\label{DD1CC2}\varphi_{C_{n/d'}}(D_{n}/C_{n/d})&=&\begin{cases} 2d' & \mbox{if}\ d'\mid d, \\ 0 & \mbox{if}\ d'\nmid d, \end{cases}\\
\label{DD1DC}\varphi_{D_{n/d'}}(D_{n}/C_{n/d})&=&0,\\
\label{DD1EC}\varphi_{D'_{n/d'}}(D_{n}/C_{n/d})&=&0,
\end{eqnarray}
and 
\begin{eqnarray}
\label{DD2CC2}\varphi_{C_{n/d'},t}(J_{N_c,A,a}(D_{n}/C_{n/d}))&=&(1+|A'|t^{[d,d']/d'})^{2(d,d')},\\
\label{DD2DC}\varphi_{D_{n/d'},t}(J_{N_c,A,a}(D_{n}/C_{n/d}))&=&(1+|A'|t^{2[d,d']/d'})^{(d,d')},\\
\label{DD2EC}\varphi_{D'_{n/d'},t}(J_{N_c,A,a}(D_{n}/C_{n/d}))&=&(1+|A'|t^{2[d,d']/d'})^{(d,d')}
\end{eqnarray}
hold. We remark that identities (\ref{DD1CC}) and (\ref{DD1CC2}), or (\ref{DD2CC}) and (\ref{DD2CC2}) are same calculation results respectively, and identities (\ref{DD2DC}) and (\ref{DD2EC}) are equal.

\item [(iii)] We consider the case of $K=D_{n/d}$ and $K'=D_{n/d'}$. For any integer $i\geq 0$, we have
\[ D_{n/d'}a^i D_{n/d}=a^i C_{n/(d,d')} \cup a^{-i} C_{n/(d,d')} \cup ba^i C_{n/(d,d')} \cup ba^{-i} C_{n/(d,d')}\]
which gives the following double cosets representation,
\begin{eqnarray*}\label{DD5}
D_{n}=\begin{cases}\bigcup_{i=0}^{(d,d')/2}D_{n/d'}a^i D_{n/d} & \mbox{if}\ 2\mid (d,d'), \\ \bigcup_{i=0}^{((d,d')-1)/2}D_{n/d'}a^i D_{n/d} & \mbox{if}\ 2\nmid (d,d').\end{cases}
\end{eqnarray*}
Next, we investigate a subgroup $D_{n/d'}\cap a^{i} D_{n/d} a^{-i}$ for any integer $i\geq 0$ to consider orbits of $\Res{D_{n}}{D_{n/d'}}(D_{n}/D_{n/d})$. A subgroup $D_{n/d'}\cap a^{i} D_{n/d} a^{-i}$ is conjugate to $C_{n/[d,d']}, D_{n/[d,d']}$ or $D'_{n/[d,d']}$. An integer $i\geq 0$ satisfies $2i\mid (d,d')$ if and only if $D_{n/d'}\cap a^{i} D_{n/d} a^{-i}$ has an element which has the form $ba^k$. \par 

If $2\mid (d,d')$, then $D_{n/d'}\cap a^{i} D_{n/d} a^{-i}$ is conjugate to $D_{n/[d,d']}$ when $i=0$ or $(d,d')/2$, and is conjugate to $C_{n/[d,d']}$ when $i\neq 0$ and $(d,d')/2$. Then, we have
\begin{eqnarray}\label{DD3DD}
 O_{D_{n}/D_{n/d}, D_{n/d'},i}=\begin{cases} 2 & \mbox{if}\ i=\frac{[d,d']}{d'},\\ \frac{(d,d')}{2}-1 & \mbox{if}\ i=\frac{2[d,d']}{d'}, \\ 0 & \mbox{otherwise}. \end{cases}
\end{eqnarray}
Hence, we have
\begin{eqnarray}
\label{DD1DD}\varphi_{D_{n/d'}}(D_{n}/D_{n/d})=\begin{cases} 2 & \mbox{if}\ d'\mid d, \\ 0& \mbox{if}\ d'\nmid d,\end{cases}
\end{eqnarray}
and by (\ref{73eq1}) and (\ref{DD3DD}), we have
\begin{eqnarray}
\begin{split}
\label{DD2DD}
&\varphi_{D_{n/d'},t}(J_{N_c,A,a}(D_{n}/D_{n/d}))\\
=&(1+|A'|t^{[d,d']/d'})^2(1+|A'|t^{2[d,d']/d'})^{((d,d')/2)-1}.
\end{split}
\end{eqnarray}
If $2\nmid (d,d')$, then $D_{n/d'}\cap a^{i} D_{n/d} a^{-i}$ is conjugate to $D_{n/[d,d']}$ when $i=0$ and is conjugate to $C_{n/[d,d']}$ when $i\neq 0$. Then, we have
\begin{eqnarray}\label{DD4DD}
 O_{D_{n}/D_{n/d}, D_{n/d'},i}=\begin{cases} 1 & \mbox{if}\ i=\frac{[d,d']}{d'},\\ \frac{(d,d')-1}{2} & \mbox{if}\ i=\frac{2[d,d']}{d'}, \\ 0 & \mbox{otherwise}. \end{cases}
\end{eqnarray}
Hence, we have
\begin{eqnarray}
\label{DD1DDD}\varphi_{D_{n/d'}}(D_{n}/D_{n/d})=\begin{cases} 1 & \mbox{if}\ d'\mid d, \\ 0& \mbox{if}\ d'\nmid d,\end{cases}
\end{eqnarray}
and by (\ref{73eq1}) and (\ref{DD4DD}), we have
\begin{eqnarray}
\begin{split}
\label{DD2DDD}
&\varphi_{D_{n/d'},t}(J_{N_c,A,a}(D_{n}/D_{n/d}))\\
=&(1+|A'|t^{[d,d']/d'})(1+|A'|t^{2[d,d']/d'})^{((d,d')-1)/2}.
\end{split}
\end{eqnarray}

\item [(iv)] We consider the case of $K=D_{n/d}$ and $K'=D'_{n/d'}$. For any integer $i\geq 0$ we have
\[ D'_{n/d'}a^i D_{n/d}=a^i C_{n/(d,d')} \cup a^{-(i+1)} C_{n/(d,d')} \cup ba^{i+1} C_{n/(d,d')} \cup ba^{-i} C_{n/(d,d')}\]
which gives the following double cosets representation.
\begin{eqnarray}\label{DD8}
D_{n}=\begin{cases}\bigcup_{i=0}^{((d,d')/2)-1}D'_{n/d'}a^i D_{n/d} & \mbox{if}\ 2\mid (d,d'), \\ \bigcup_{i=0}^{((d,d')-1)/2}D'_{n/d'}a^i D_{n/d} & \mbox{if}\ 2\nmid (d,d').\end{cases}
\end{eqnarray}
A subgroup $D'_{n/d'}\cap a^{i} D_{n/d} a^{-i}$ is conjugate to either $C_{n/[d,d']}$, $D_{n/[d,d']}$ or $D'_{n/[d,d']}$. An integer $i\geq 0$ satisfies $2i+1\mid (d,d')$ if and only if $D_{n/d'}\cap a^{i} D_{n/d} a^{-i}$ has an element which has the form $ba^k$. \par 
If $2\mid (d,d')$, then there exists no integer $i\geq 0$ such that $D'_{n/d'}\cap a^{i} D_{n/d} a^{-i}$ is conjugate to $D_{n/[d,d']}$ or $D'_{n/[d,d']}$. Then, we have
\begin{eqnarray}\label{DD3DE}
O_{D_{n}/D_{n/d}, D_{n/d'}',i}=\begin{cases} \frac{(d,d')}{2} & \mbox{if}\ i=\frac{2[d,d']}{d'}, \\ 0 & \mbox{otherwise} \end{cases}
\end{eqnarray}
which is equal to (\ref{DD4DD}). Hence, we have
\begin{eqnarray}\label{DD1DE}
\varphi_{D'_{n/d'}}(D_{n}/D_{n/d})=0,
\end{eqnarray}
and by (\ref{73eq1}) and (\ref{DD3DE}), we have
\begin{eqnarray}\label{DD2DE}
\begin{split}
\varphi_{D'_{n/d'},t}(J_{N_c,A,a}(D_{n}/D_{n/d}))=(1+|A'|t^{2[d,d']/d'})^{(d,d')/2}.
\end{split}
\end{eqnarray}

If $2\nmid (d,d')$, then $D_{n/d'}\cap a^{i} D_{n/d} a^{-i}$ is conjugate to $D_{n/[d,d']}$ when $i=0$ and is conjugate to $C_{n/[d,d']}$ when $i\neq 0$. Then, we have
\begin{eqnarray}\label{DD4DE}
 O_{D_{n}/D_{n/d}, D'_{n/d'},i}=\begin{cases} 1 & \mbox{if}\ i=\frac{[d,d']}{d'},\\ \frac{(d,d')-1}{2} & \mbox{if}\ i=\frac{2[d,d']}{d'}, \\ 0 & \mbox{otherwise}. \end{cases}
\end{eqnarray}
Hence, we have
\begin{eqnarray}\label{DD1DEE}
\varphi_{D'_{n/d'}}(D_{n}/D_{n/d})=\begin{cases} 1 &\mbox{if}\ d'\mid d,\\ 0&\mbox{if}\ d'\nmid d,\end{cases}
\end{eqnarray}
which is equal to (\ref{DD1DDD}), and by (\ref{73eq1}) and (\ref{DD4DE}), we have
\begin{eqnarray}\label{DD2DEE}
\begin{split}
&\varphi_{D'_{n/d'},t}(J_{N_c,A,a}(D_{n}/D_{n/d}))\\
=&(1+|A'|t^{[d,d']/d'})(1+|A'|t^{2[d,d']/d'})^{((d,d')-1)/2}
\end{split}
\end{eqnarray}
which is equal to (\ref{DD2DDD}).

\item [(v)] We consider the case of $K=D'_{n/d}$ and $K'=D_{n/d'}$ with the method similar to (iv). By (\ref{DD8}), we have
\begin{eqnarray}\label{DD8}
D_{n}=\begin{cases}\bigcup_{i=0}^{((d,d')/2)-1}D_{n/d'}a^{-i} D'_{n/d} & \mbox{if}\ 2\mid (d,d'), \\ \bigcup_{i=0}^{((d,d')-1)/2}D_{n/d'}a^{-i} D'_{n/d} & \mbox{if}\ 2\nmid (d,d').\end{cases}
\end{eqnarray}
If $2\mid (d,d')$, then there exists no integer $i\geq 0$ such that $D'_{n/d}\cap a^{-i} D_{n/d'} a^{i}$ is conjugate to $D_{n/[d,d']}$ or $D'_{n/[d,d']}$. Then, we have
\begin{eqnarray}\label{DD3ED}
O_{D_{n}/D'_{n/d}, D_{n/d'},i}=\begin{cases} \frac{(d,d')}{2} & \mbox{if}\ i=\frac{2[d,d']}{d'}, \\ 0 & \mbox{otherwise}. \end{cases}
\end{eqnarray}
Hence, we have
\begin{eqnarray}\label{DD1ED}
\varphi_{D_{n/d'}}(D_{n}/D'_{n/d})=0,
\end{eqnarray}
and by (\ref{73eq1}) and (\ref{DD3ED}), we have
\begin{eqnarray}
\begin{split}\label{DD2ED}
\varphi_{D_{n/d'},t}(J_{N_c,A,a}(D_{n}/D'_{n/d}))=(1+|A'|t^{2[d,d']/d'})^{(d,d')/2}.
\end{split}
\end{eqnarray}
If $2\nmid (d,d')$, then $D_{n/d'}\cap a^{i} D_{n/d} a^{-i}$ is conjugate to $D_{n/[d,d']}$ when $i=0$ and is conjugate to $C_{n/[d,d']}$ when $i\neq 0$. Thus, we have
\begin{eqnarray}\label{DD4ED}
 O_{D_{n}/D'_{n/d}, D_{n/d'},i}=\begin{cases} 1 & \mbox{if}\ i=\frac{[d,d']}{d'},\\ \frac{(d,d')-1}{2} & \mbox{if}\ i=\frac{2[d,d']}{d'}, \\ 0 & \mbox{otherwise}. \end{cases}
\end{eqnarray}
Hence, we have
\begin{eqnarray}\label{DD1EDD}
\varphi_{D_{n/d'}}(D_{n}/D_{n/d})=\begin{cases} 1 &\mbox{if}\ d'\mid d,\\ 0&\mbox{if}\ d'\nmid d,\end{cases}
\end{eqnarray}
and by (\ref{73eq1}) and (\ref{DD4ED}), we have
\begin{eqnarray}
\begin{split}
\label{DD2EDD}
&\varphi_{D_{n/d'},t}(J_{N_c,A,a}(D_{n}/D'_{n/d}))\\
=&(1+|A'|t^{[d,d']/d'})(1+|A'|t^{2[d,d']/d'})^{((d,d')-1)/2}.
\end{split}
\end{eqnarray}

\item [(vi)] We assume $2\mid n$, and we consider when $K=D'_{n/d}$ and $K'=D'_{n/d'}$. For any integer $i\geq 0$ we have\[ D'_{n/d'}a^i D'_{n/d}=a^i C_{n/(d,d')} \cup a^{-i} C_{n/(d,d')} \cup ba^{i+1} C_{n/(d,d')} \cup ba^{-i+1} C_{n/(d,d')}\]
which gives the following double cosets representation,
\begin{eqnarray}\label{DD11}
D_{n}=\begin{cases}\bigcup_{i=0}^{(d,d')/2}D'_{n/d'}a^i D'_{n/d} & \mbox{if}\ 2\mid (d,d'), \\ \bigcup_{i=0}^{((d,d')-1)/2}D'_{n/d'}a^i D'_{n/d} & \mbox{if}\ 2\nmid (d,d').\end{cases}
\end{eqnarray}
A subgroup $D'_{n/d'}\cap a^{i} D'_{n/d} a^{-i}$ is conjugate to $C_{n/[d,d']}$, $D_{n/[d,d']}$ or $D'_{n/[d,d']}$. An integer $i\geq 0$ satisfies $2i\mid (d,d')$ if and only if $D'_{n/d'}\cap a^{i} D'_{n/d} a^{-i}$ has an element which has a form $ba^k$. Then we have
\begin{eqnarray}\label{DD4EE}
 O_{D_{n}/D'_{n/d}, D_{n/d'},i}=\begin{cases} 2 & \mbox{if}\ i=\frac{[d,d']}{d'},\\ \frac{(d,d')}{2}-1 & \mbox{if}\ i=\frac{2[d,d']}{d'}, \\ 0 & \mbox{otherwise}. \end{cases}
\end{eqnarray}
Hence, we have
\begin{eqnarray}\label{DD1EE}
\varphi_{D'_{n/d'}}(D'_{n}/D'_{n/d})=\begin{cases} 2 & \mbox{if}\ d'\mid d, \\ 0 & \mbox{if}\ d'\nmid d,\end{cases} 
\end{eqnarray}
which is equal to (\ref{DD1DD}), and by (\ref{73eq1}) and (\ref{DD4EE}), we have
\begin{eqnarray}
\begin{split}
\label{DD2EE}
&\varphi_{D'_{n/d'},t}(J_{N_c,A,a}(D_{n}/D'_{n/d}))\\
=&(1+|A'|t^{[d,d']/d'})^2(1+|A'|t^{2[d,d']/d'})^{((d,d')/2)-1}
\end{split}
\end{eqnarray}
which is equal to (\ref{DD2DD}).
\end{itemize}
Next, we represent $\mu_{H}(X)$ with $\varphi_{V}(X)$ for any $D_{n}$-set $X$. 
\begin{prop}\label{DDD2}
For any $D_{n}$-set $X$ and divisor $d$ of $n$, the followings hold,\\
$(1)$
\begin{eqnarray*}
&{}&\dfrac{1}{d}\sum_{d'\mid d}\mu\Big(\dfrac{d}{d'}\Big)\varphi_{C_{n/d'}}(X)\\
&=&\begin{cases} 2\mu_{C_{n/d}}(X)+\mu_{D_{n/d}}(X) & \mbox{if}\ 2\nmid d,\\ 2\mu_{C_{n/d}}(X)+\mu_{D_{n/d}}(X)+\mu_{D'_{n/d}}(X) & \mbox{if}\ 2\mid d.\end{cases}
\end{eqnarray*}
$(2)$
\begin{eqnarray*}
\sum_{d'\mid d}\mu\Big(\dfrac{d}{d'}\Big)\varphi_{D_{n/d'}}(X)=\begin{cases} \mu_{D_{n/d}}(X) & \mbox{if}\ 2\nmid d, \\ 2\mu_{D_{n/d}}(X)& \mbox{if}\ 2\mid d. \end{cases}
\end{eqnarray*}
$(3)$ If $2\mid d$, then 
\begin{eqnarray*}
\sum_{d'\mid d}\mu\Big(\dfrac{d}{d'}\Big)\varphi_{D'_{n/d'}}(X)=2\mu_{D'_{n/d}}(X)
\end{eqnarray*}
holds.
\end{prop}
\begin{proof}
By Proposition \ref{CongDn}, an arbitrary $D_{n}$-set $X$ has the following form,
\begin{eqnarray}\label{DDD1}
\begin{split}
[X]=&\sum_{d'\mid n, 2\nmid d'}\Big( \mu_{C_{n/d'}}(X)[D_{n}/C_{n/d'}]+\mu_{D_{n/d'}}(X)[D_{n}/D_{n/d'}]\Big)\\
&+ \sum_{d'\mid n, 2\mid d'} \Big( \mu_{C_{n/d'}}(X)[D_{n}/C_{n/d'}]+ \mu_{D_{n/d'}}(X) [D_{n}/D_{n/d'}]\\
&+ \mu_{D_{n/d'}}(X)[D_{n}/D'_{n/d'}]\Big).
\end{split}
\end{eqnarray}
By (\ref{DDD1}), (\ref{DD1CC}), (\ref{DD1CD}) and (\ref{DD1CE}), we have
\begin{eqnarray}
\begin{split}\label{DDD21}
\varphi_{C_{n/d}}(X)=&\sum_{d'\mid d, 2\nmid d'} (2d'\mu_{C_{n/d'}}(X) + d'\mu_{D_{n/d'}}(X))\\
&+\sum_{d'\mid d, 2\mid d'} (2d'\mu_{C_{n/d'}}(X)+ d'\mu_{D_{n/d'}}(X)+d'\mu_{D_{n/d'}}(X)).
\end{split}
\end{eqnarray}
Thus, we have the identity $(1)$ by (\ref{DDD21}) and the \Mobius\ inversion formula.

By (\ref{DDD1}), (\ref{DD1DC}), (\ref{DD1DD}), (\ref{DD1DDD}), (\ref{DD1ED}) and (\ref{DD1EDD}) we have
\begin{eqnarray}\label{DDD22}
\varphi_{D_{n/d}}(X)&=&\sum_{d'\mid d, 2\nmid d'}\mu_{D_{n/d'}}(X)+\sum_{d'\mid d, 2\mid d'}2\mu_{D_{n/d'}}(X).
\end{eqnarray}
Thus, we have the identity $(2)$ by (\ref{DDD22}) and the \Mobius\ inversion formula.

By (\ref{DDD1}), (\ref{DD1EC}), (\ref{DD1ED}), (\ref{DD1EDD}) and (\ref{DD1EE}) we have 
\begin{eqnarray}\label{DDD23}
\varphi_{D_{n/d'}}(X)&=&\sum_{d'\mid d, 2\nmid d'}\mu_{D_{n/d'}}(X)+\sum_{d'\mid d, 2\mid d'}2\mu_{D'_{n/d'}}(X).
\end{eqnarray}
Thus, we have the identity $(3)$ by (\ref{DDD23}) and the \Mobius\ inversion formula.
\end{proof}

By Proposition \ref{DDD2}, we have the following proposition.
\begin{prop}\label{DDDD2}
For any $D_{n}$-set $X$ and divisor $d$ of $n$, the followings hold,\\
$(1)$
\begin{eqnarray*}
&{}&\dfrac{1}{d}\sum_{d'\mid d}\mu\Big(\dfrac{d}{d'}\Big)\varphi_{C_{n/d'},t}(J_{N_c,A,a}(X))\\
&=&\begin{cases} 2\mu_{C_{n/d},t}(J_{N_c,A,a}(X))+\mu_{D_{n/d},t}(J_{N_c,A,a}(X)) & \mbox{if}\ 2\nmid d,\\ 2\mu_{C_{n/d},t}(J_{N_c,A,a}(X))\\
+\mu_{D_{n/d},t}(J_{N_c,A,a}(X))+\mu_{D'_{n/d},t}(J_{N_c,A,a}(X)) & \mbox{if}\ 2\mid d.\end{cases}
\end{eqnarray*}
$(2)$
\begin{eqnarray*}
\sum_{d'\mid d}\mu\Big(\dfrac{d}{d'}\Big)\varphi_{D_{n/d'},t}(J_{N_c,A,a}(X))=\begin{cases} \mu_{D_{n/d},t}(J_{N_c,A,a}(X)) & \mbox{if}\ 2\nmid d, \\ 2\mu_{D_{n/d},t}(J_{N_c,A,a}(X))& \mbox{if}\ 2\mid d. \end{cases}
\end{eqnarray*}
$(3)$ If $2\mid d$, then 
\begin{eqnarray*}
&{}&\sum_{d'\mid d}\mu\Big(\dfrac{d}{d'}\Big)\varphi_{D'_{n/d'},t}(J_{N_c,A,a}(X))=2\mu_{D'_{n/d},t}(J_{N_c,A,a}(X))
\end{eqnarray*}
holds.
\end{prop}
Note that $(2)$ and $(3)$ are equal if $2\mid d$ holds.
\begin{proof}
To proof this Proposition \ref{DDDD2}, we consider the coefficient of $t^n$ in these elements and use Proposition \ref{DDD2}.

(1) If $2\nmid d$, we have
\begin{eqnarray*}
&{}&\dfrac{1}{d}\sum_{d'\mid d}\mu\Big(\dfrac{d}{d'}\Big)\varphi_{C_{n/d'},t}(J_{N_c,A,a}(X))\\
&=&\dfrac{1}{d}\sum_{d'\mid d}\sum_{i=0}^{\infty}\mu\Big(\dfrac{d}{d'}\Big)\varphi_{C_{n/d'}}(J^i_{N_c,A,a}(X))t^i\\
&=&\sum_{i=0}^{\infty}(2\mu_{C_{n/d}}(J^i_{N_c,A,a}(X))+\mu_{D_{n/d}}(J^i_{N_c,A,a}(X)))t^i\\
&=&2\mu_{C_{n/d},t}(J_{N_c,A,a}(X))+\mu_{D_{n/d},t}(J_{N_c,A,a}(X)).
\end{eqnarray*}
Similarly, if $2\mid d$, we have
\begin{eqnarray*}
&{}&\dfrac{1}{d}\sum_{d'\mid d}\mu\Big(\dfrac{d}{d'}\Big)\varphi_{C_{n/d'},t}(J_{N_c,A,a}(X))\\
&=&\dfrac{1}{d}\sum_{d'\mid d}\sum_{i=0}^{\infty}\mu\Big(\dfrac{d}{d'}\Big)\varphi_{C_{n/d'}}(J^i_{N_c,A,a}(X))t^i\\
&=&\sum_{i=0}^{\infty}(2\mu_{C_{n/d'}}(J^i_{N_c,A,a}(X))+\mu_{D_{n/d}}(J^i_{N_c,A,a}(X))+\mu_{D'_{n/d}}(J^i_{N_c,A,a}(X)))t^i\\
&=&2\mu_{C_{n/d},t}(J_{N_c,A,a}(X))+\mu_{D_{n/d},t}(J_{N_c,A,a}(X))+\mu_{D_{n/d'},t}(J_{N_c,A,a}(X)).
\end{eqnarray*}

(2) If $2\nmid d$, we have
\begin{eqnarray*}
&{}&\dfrac{1}{d}\sum_{d'\mid d}\mu\Big(\dfrac{d}{d'}\Big)\varphi_{D_{n/d'},t}(J_{N_c,A,a}(X))\\
&=&\dfrac{1}{d}\sum_{d'\mid d}\sum_{i=0}^{\infty}\mu\Big(\dfrac{d}{d'}\Big)\varphi_{D_{n/d'}}(J^i_{N_c,A,a}(X))t^i\\
&=&\sum_{i=0}^{\infty}\mu_{D_{n/d}}(J^i_{N_c,A,a}(X))t^i=\mu_{D_{n/d},t}(J_{N_c,A,a}(X)).
\end{eqnarray*}
Similarly, if $2\mid d$, we have
\begin{eqnarray*}
&{}&\dfrac{1}{d}\sum_{d'\mid d}\mu\Big(\dfrac{d}{d'}\Big)\varphi_{D_{n/d'},t}(J_{N_c,A,a}(X))\\
&=&\dfrac{1}{d}\sum_{d'\mid d}\sum_{i=0}^{\infty}\mu\Big(\dfrac{d}{d'}\Big)\varphi_{D_{n/d'}}(J^i_{N_c,A,a}(X))t^i\\
&=&\sum_{i=0}^{\infty}2\mu_{D_{n/d}}(J^i_{N_c,A,a}(X))t^i=2\mu_{D_{n/d},t}(J_{N_c,A,a}(X)).
\end{eqnarray*}
(3) We assume $2\mid d$. Thus, we have
\begin{eqnarray*}
&{}&\dfrac{1}{d}\sum_{d'\mid d}\mu\Big(\dfrac{d}{d'}\Big)\varphi_{D'_{n/d'},t}(J_{N_c,A,a}(X))\\
&=&\dfrac{1}{d}\sum_{d'\mid d}\sum_{i=0}^{\infty}\mu\Big(\dfrac{d}{d'}\Big)\varphi_{D'_{n/d'}}(J^i_{N_c,A,a}(X))t^i\\
&=&\sum_{i=0}^{\infty}2\mu_{D'_{n/d}}(J^i_{N_c,A,a}(X))t^i=2\mu_{D'_{n/d},t}(J_{N_c,A,a}(X)).
\end{eqnarray*}
\end{proof}

By Proposition \ref{DDDD2}, we have the following corollary.
\begin{cor}\label{DDDDD2}
For any $D_{n}$-set $X$ and divisor $d$ of $n$, the followings hold,\\
$(1)$ If $2\nmid d$, the following identities hold,
\begin{eqnarray}
\label{CAL1D}&{}&\mu_{D_{n/d},t}(J_{N_c,A,a}(X))=\sum_{d'\mid d}\mu\Big(\dfrac{d}{d'}\Big)\varphi_{D_{n/d'},t}(J_{N_c,A,a}(X)),\\
\label{CAL1CC}&{}&\mu_{C_{n/d},t}(J_{N_c,A,a}(X))\\
\nonumber &=&\sum_{d'\mid d}\mu\Big(\dfrac{d}{d'}\Big)\Big(\dfrac{1}{2d}\varphi_{C_{n/d'},t}(J_{N_c,A,a}(X))-\dfrac{1}{2}\varphi_{D_{n/d'},t}(J_{N_c,A,a}(X))\Big).
\end{eqnarray}
$(2)$ If $2\mid d$, the following identities hold,
\begin{eqnarray}
\label{CAL2D}&{}&\mu_{D_{n/d},t}(J_{N_c,A,a}(X))=\dfrac{1}{2}\sum_{d'\mid d}\mu\Big(\dfrac{d}{d'}\Big)\varphi_{D_{n/d'},t}(J_{N_c,A,a}(X)),\\
\label{CAL2DD}&{}&\mu_{D'_{n/d},t}(J_{N_c,A,a}(X))=\dfrac{1}{2}\sum_{d'\mid d}\mu\Big(\dfrac{d}{d'}\Big)\varphi_{D'_{n/d'},t}(J_{N_c,A,a}(X)),\\
\nonumber&{}&\mu_{C_{n/d},t}(J_{N_c,A,a}(X))\\
\label{CAL2C}&=&\dfrac{1}{2d}\sum_{d'\mid d}\mu\Big(\dfrac{d}{d'}\Big)\varphi_{C_{n/d'},t}(J_{N_c,A,a}(X))-\dfrac{1}{2}\mu_{D_{n/d},t}(J_{N_c,A,a}(X))\\
\nonumber&{}&-\dfrac{1}{2}\mu_{D_{n/d},t}(J_{N_c,A,a}(X))\\
\label{CAL2CC}&=&\sum_{d'\mid d}\mu\Big(\dfrac{d}{d'}\Big)\Big(\dfrac{1}{2d}\varphi_{C_{n/d'},t}(J_{N_c,A,a}(X)) -\dfrac{1}{4}\varphi_{D_{n/d'},t}(J_{N_c,A,a}(X))\\
\nonumber&{}&-\dfrac{1}{4}\varphi_{D'_{n/d'},t}(J_{N_c,A,a}(X))\Big).
\end{eqnarray}
\end{cor}
\begin{proof}
We assume $2\nmid d$. The identity about $\mu_{D_{n/d},t}(J_{N_c,A,a}(X))$ holds since Proposition \ref{DDDD2} (2) holds. For $\mu_{C_{n/d},t}(J_{N_c,A,a}(X))$, substitute the identity of Proposition \ref{DDDD2} (2) in (1). Then, we have the identity (\ref{CAL1CC}).

Next, we assume $2\mid d$. The identities (\ref{CAL2DD}) or (\ref{CAL2DD}) holds by Proposition \ref{DDDD2} (2) or (3), respectively. Then, by Proposition \ref{DDDD2} (1), (\ref{CAL2D})and (\ref{CAL2DD}), we have the identity (\ref{CAL2C}) and (\ref{CAL2CC}).
\end{proof}

\subsubsection{The set of all vertices of a regular $n$-prism}
Let $X_n'$ be the set of all vertices of a regular $n$-prism. We can regard the set $X_n'$ as a $D_{n}$-set, and $X_n'$ is $D_{n}$-isomorphic to $D_{n}/C_1$. Substituting $d=n$ in (\ref{DD2CC}), (\ref{DD2DC}) and (\ref{DD2EC}), we have $n/n=1$ and
\begin{eqnarray*}
\varphi_{C_{n/d'},t}(J_{N_c,A,a}(D_{n}/C_{1}))&=&(1+|A'|t^{n/d'})^{2d'},\\
\varphi_{D_{n/d'},t}(J_{N_c,A,a}(D_{n}/C_{1}))&=&(1+|A'|t^{2n/d'})^{d'},\\
\varphi_{D'_{n/d'},t}(J_{N_c,A,a}(D_{n}/C_{1}))&=&(1+|A'|t^{2n/d'})^{d'}.
\end{eqnarray*}
 Let $d$ be a divisor of $n$. If $2\nmid d$, then we substitute these identities in (\ref{CAL1CC}). Then, we have
\begin{eqnarray}\label{Prism1}
\begin{split}
{}&\mu_{C_{n/d},t}(J_{N_c,A,a}(D_{n}/C_{1}))\\
=&\sum_{d'\mid d}\mu\Big(\dfrac{d}{d'}\Big)\Big(\dfrac{1}{2d}(1+|A'|t^{n/d'})^{2d'}-\dfrac{1}{2}(1+|A'|t^{2n/d'})^{d'}\Big).
\end{split}
\end{eqnarray}
If $2\mid d$, them we substitute in (\ref{CAL2C}). Then, we have
\begin{eqnarray*}
&{}&\mu_{C_{n/d},t}(J_{N_c,A,a}(D_{n}/C_{1}))\\
&=&\sum_{d'\mid d}\mu\Big(\dfrac{d}{d'}\Big)\Big(\dfrac{1}{2d}(1+|A'|t^{n/d'})^{2d'}-\dfrac{1}{4}(1+|A'|t^{2n/d'})^{d'}\\
&{}&-\dfrac{1}{4}(1+|A'|t^{2n/d'})^{d'}\Big)\\
&=&\sum_{d'\mid d}\mu\Big(\dfrac{d}{d'}\Big)\Big(\dfrac{1}{2d}(1+|A'|t^{n/d'})^{2d'}-\dfrac{1}{2}(1+|A'|t^{2n/d'})^{d'}\Big).\\
\end{eqnarray*}
In any case, we have (\ref{Prism1}) for any divisor $d$ of $n$. Substituting $d=n$ in (\ref{Prism1}), we have 
\begin{eqnarray}\label{Prism2}
\begin{split}
{}&\mu_{C_{1},t}(J_{N_c,A,a}(D_{n}/C_{1}))\\
=&\sum_{d\mid n}\mu\Big(\dfrac{n}{d}\Big)\Big(\dfrac{1}{2n}(1+|A'|t^{n/d})^{2d}-\dfrac{1}{2}(1+|A'|t^{2n/d})^{d}\Big)
\end{split}
\end{eqnarray}
and substituting $t=1$ in (\ref{Prism2}), we have
\begin{eqnarray*}
&{}&\mu_{C_{1}}(J_{N_c,A,a}(D_{n}/C_{1}))\\
&=&\dfrac{1}{2}\sum_{d\mid n}\mu\Big(\dfrac{n}{d}\Big)\Big(\dfrac{1}{n}|A|^{2d}-|A|^{d}\Big)=\dfrac{1}{2}(M(|A|^2,n)-nM(|A|,n))
\end{eqnarray*}
by Lemma \ref{73eqlemm2}, which is the number of primitive colorings on the set of all vertices of the $n$-prism.
\begin{exa}
We consider the case of $A=\{0,1\}$, $a=0$ and $n=3$. Substitute these condition and $d=n$ in (\ref{Prism2}), we have 
\begin{eqnarray*}
&{}&\mu_{C_{1},t}(J_{N_c,A,a}(D_3/C_1))\\
&=&\dfrac{1}{2}\Big(-\dfrac{1}{3}(1+t^3)^2+(1+t^6)+\dfrac{1}{3}(1+t)^6-(1+t^2)^3\Big)\\
&=&t+t^2+3t^3+t^4+t^5.
\end{eqnarray*}
In particular, we have $\mu_{C_{1}}(J_{N_c,A,a}(D_3/C_1))=7$. All results of placing $2$ colored on the vertices of a regular $3$-prism are described in Figure 4.
\begin{figure}[h]\label{Img3}
	\centering
	\includegraphics[width=10cm,height=10cm]{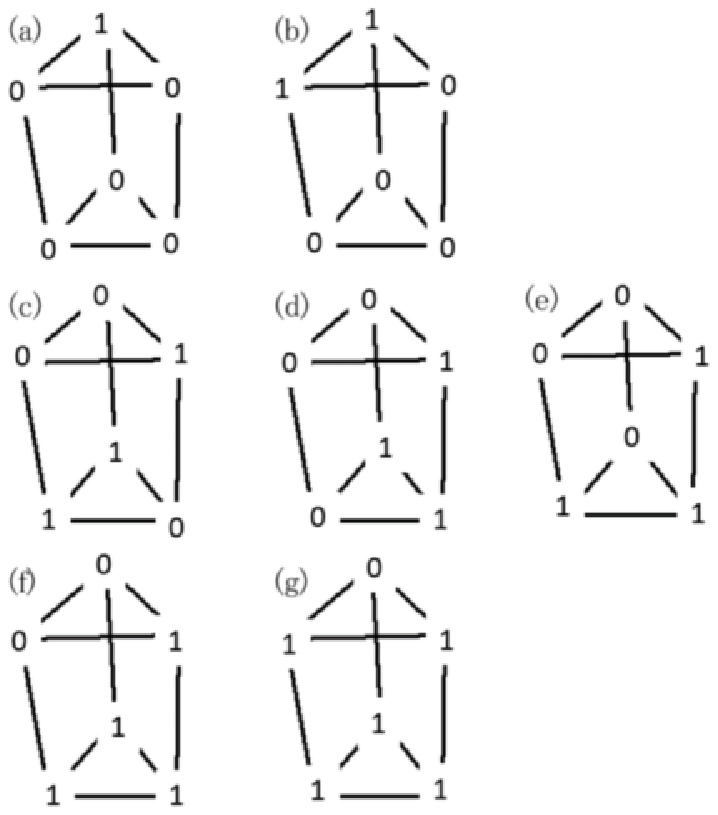}\\
	Figure 4: Primitive colorings on vertices of a regular 3-prism of $D_3$ with $2$-colors $\{0,1\}$.
\end{figure}
In Figure 4, the image (a), (b), or (f) is the unique primitive coloring which has five, four or two $0$-colored places, respectively, images (c), (d) and (e) are primitive colorings which have $3$-times $0$-colored places, and the image (g) is the unique primitive coloring which has one $0$-colored place.
\end{exa}

\subsubsection{The set of all vertices of a regular $n$-gon with the dihedral group}
Let $X_n$ be the set of all vertices of a $n$-gon. In \S 4.1, we considered cycle rotations. In addition, we consider cycle rotations and reverses of a regular $n$-gon. We assume that the set $X_n$ is a $D_{n}$-set which is $D_{n}$-isomorphic to $D_{n}/D_1$, and we calculate $\mu_{C_1,t}(J_{N_c,A,a}(D_{n}/D_1))$ and $\mu_{C_1}(J_{N_c,A,a}(D_{n}/D_1))$.

In the remainder of this section, we define a polynomial $M(k,n)\in\mathbb{C}[k]$ by (\ref{necklace}). There are cases that we substitute $k=\sqrt{A}$ in $M(k,n)$.

Substituting $d=n$ in (\ref{DD2CD}), we have
\[\varphi_{C_{n/d'},t}(J_{N_c,A,a}(D_{n}/D_{1}))=(1+|A'|t^{n/d'})^{d'}.\]
A divisor $d'$ of $n$ satisfies $2\mid d'$ if and only if $2\mid (n,d')$. Thus, if $2\mid d'$, then we have
\begin{eqnarray*}
\varphi_{D_{n/d'},t}(J_{N_c,A,a}(D_{n}/D_{1}))&=&(1+|A'|t^{n/d'})^2 (1+|A'|t^{2n/d'})^{(d'/2)-1},\\
\varphi_{D'_{n/d'},t}(J_{N_c,A,a}(D_{n}/D_{1}))&=&(1+|A'|t^{2n/d'})^{d'/2},
\end{eqnarray*}
since we substitute $d=n$ in (\ref{DD2DD}) and (\ref{DD2DE}).

If $2\nmid d'$, then we have
\begin{eqnarray*}
\varphi_{D_{n/d'},t}(J_{N_c,A,a}(D_{n}/D_{1}))&=&(1+|A'|t^{n/d'})(1+|A'|t^{2n/d'})^{(d'-1)/2},\\
\varphi_{D'_{n/d'},t}(J_{N_c,A,a}(D_{n}/D_{1}))&=&(1+|A'|t^{n/d'})(1+|A'|t^{2n/d'})^{(d'-1)/2},
\end{eqnarray*}
since we substitute $d=n$ in (\ref{DD2DDD}) and (\ref{DD2DEE}).

Substituting these five identities in identities of Corollary \ref{DDDDD2}, we calculate $\mu_{C_{n/d}}(J_{N_c,A,a}(D_n/D_1))$ for any divisor $d$ of $n$. 
\begin{itemize}
\item[{\rm (I) }] If $2\nmid d$, then we have 
\begin{eqnarray}\label{DR1}
\begin{split}
&\mu_{C_{n/d},t}(J_{N_c,A,a}(D_{n}/D_{1}))\\
=&\sum_{d'\mid d}\mu\Big(\dfrac{d}{d'}\Big)\Big(\dfrac{1}{2d}(1+|A'|t^{n/d'})^{d'}\\
&-\dfrac{1}{2}(1+|A'|t^{n/d'})(1+|A'|t^{2n/d'})^{(d'-1)/2}\Big)
\end{split}
\end{eqnarray}
by (\ref{CAL1CC}). Substituting $t=1$, we have 
\begin{eqnarray}\label{DR11}
\begin{split}
&\mu_{C_{n/d}}(J_{N_c,A,a}(D_{n}/D_{1}))\\
=&\sum_{d'\mid d}\mu\Big(\dfrac{d}{d'}\Big)\Big(\dfrac{1}{2d}|A|^{d'}-\dfrac{1}{2}|A|^{(d'+1)/2}\Big)\\
=&\dfrac{1}{2}\Big(M(|A|,d)-d\sqrt{|A|}M(\sqrt{|A|},d)\Big)
\end{split}
\end{eqnarray}
by Lemma \ref{73eqlemm2}.
\item [{\rm (II)}] If $2\mid d$ and $4\nmid d$, then we have
\begin{eqnarray*}
&{}&\mu_{D_{n/d},t}(J_{N_c,A,a}(D_{n}/D_{1}))\\
&=&\dfrac{1}{2}\sum_{d'\mid d}\mu\Big(\dfrac{d}{ d'}\Big)\varphi_{D_{n/d'},t}(J_{N_c,A,a}(D_n/D_1))\\
&=&\dfrac{1}{2}\sum_{d'\mid d/2}\mu\Big(\dfrac{d}{2d'}\Big)\Big(-\varphi_{D_{n/d'},t}(J_{N_c,A,a}(D_n/D_1))\\
&{}&+\varphi_{D_{n/2d'},t}(J_{N_c,A,a}(D_n/D_1))\Big)\\
&=&\dfrac{1}{2}\sum_{d'\mid d/2}\mu\Big(\dfrac{d}{2d'}\Big)\Big(-(1+|A'|t^{n/d'})(1+|A'|t^{2n/d'})^{(d'-1)/2}\\
&{}&+(1+|A'|t^{n/2d'})^2(1+|A'|t^{n/d'})^{d'-1}\Big)
\end{eqnarray*}
and
\begin{eqnarray*}
&{}&\mu_{D'_{n/d},t}(J_{N_c,A,a}(D_{n}/D_{1}))\\
&=&\dfrac{1}{2}\sum_{d'\mid d}\mu\Big(\dfrac{d}{d'}\Big)\varphi_{D'_{n/d'},t}(J_{N_c,A,a}(D_n/D_1))\\
&=&\dfrac{1}{2}\sum_{d'\mid d/2}\mu\Big(\dfrac{d}{2d'}\Big)\Big((-\varphi_{D'_{n/d'},t}(J_{N_c,A,a}(D_n/D_1))\\
&{}&+\varphi_{D'_{n/2d'},t}(J_{N_c,A,a}(D_n/D_1))\Big)\\
&=&\dfrac{1}{2}\sum_{d'\mid d/2}\mu\Big(\dfrac{d}{2d'}\Big)\Big(-(1+|A'|t^{n/d'})(1+|A'|t^{2n/d'})^{(d'-1)/2}\\
&{}&+(1+|A'|t^{n/d'})^{d'}\Big)
\end{eqnarray*}
by (\ref{CAL2D}). Thus, we have
\begin{eqnarray}\label{DR2}
\begin{split}
&\mu_{C_{n/d},t}(J_{N_c,A,a}(D_{n}/D_{1}))\\
=&\dfrac{1}{2d}\sum_{d'\mid d}\mu\Big(\dfrac{d}{d'}\Big)(1+|A'|t^{n/d'})^{d'}\\
&-\dfrac{1}{4}\sum_{d'\mid d/2}\mu\Big(\dfrac{d}{2d'}\Big)\Big(-2(1+|A'|t^{n/d'})(1+|A'|t^{2n/d'})^{(d'-1)/2}\\
&+(1+|A|t^{n/2d'})^2(1+|A'|t^{n/d'})^{d'-1}+(1+|A'|t^{n/d'})^{d'}\Big)
\end{split}
\end{eqnarray}
by (\ref{CAL2C}). Substituting $t=1$, we have
\begin{eqnarray}\label{DR21}
\begin{split}
&\mu_{C_{n/d}}(J_{N_c,A,a}(D_{n}/D_{1}))\\
=&\dfrac{1}{2d}\sum_{d'\mid d}\mu\Big(\dfrac{d}{d'}\Big)|A|^{d'}\\
&-\dfrac{1}{4}\sum_{d'\mid d/2}\mu\Big(\dfrac{d}{2d'}\Big)(-2|A|^{(d'+1)/2}+|A|^{d'+1}+|A|^{d'})\\
=&\dfrac{1}{2}M(|A|,d)-\dfrac{d}{8}(|A|+1)M(|A|,\dfrac{d}{2})+\dfrac{d}{4}\sqrt{|A|}M(\sqrt{|A|},\dfrac{d}{2})
\end{split}
\end{eqnarray}
by Lemma \ref{73eqlemm2}.
\item [{\rm (III)}] If $4\mid d$, then all divisor $d'$ of $d$ such that $\mu(d/d')\neq 0$ holds satisfy $2\mid d'$. Then, we have
\begin{eqnarray}\label{DR3}
\begin{split}
&\mu_{C_{n/d},t}(J_{N_c,A,a}(D_{n}/D_{1}))\\
=&\sum_{d'\mid d}\mu\Big(\dfrac{d}{d'}\Big)\Big(\dfrac{1}{2d}(1+|A'|t^{n/d'})^{d'}\\
&-\dfrac{1}{4}(1+|A'|t^{n/d'})^2(1+|A'|t^{2n/d'})^{(d'/2)-1}\\
&-\dfrac{1}{4}(1+|A'|t^{2n/d'})^{d'/2}\Big)
\end{split}
\end{eqnarray}
by (\ref{CAL2CC}). Substituting $t=1$, we have
\begin{eqnarray}\label{DR31}
\begin{split}
&\mu_{C_{n/d}}(J_{N_c,A,a}(D_{n}/D_{1}))\\
=&\sum_{d'\mid d}\mu\Big(\dfrac{d}{d'}\Big)\Big(\dfrac{1}{2d}|A|^{d'}-\dfrac{1}{4}|A|^{(d'/2)+1}-\dfrac{1}{4}|A|^{d'/2}\Big)\\
=&\dfrac{1}{2}M(|A|,d)-\dfrac{d}{4}(|A|+1)M(\sqrt{|A|},d)
\end{split}
\end{eqnarray}
by Lemma \ref{73eqlemm2}.
\end{itemize}

Substituting $d=n$ in (\ref{DR11}), (\ref{DR21}) and (\ref{DR31}), we have the following results.
\begin{itemize}
\item [{\rm (I)}] If $2\nmid n$, we have
\[\mu_{C_1}(J_{N_c,A,a}(D_n/D_1))=\dfrac{1}{2}\Big(M(|A|,n)-n\sqrt{|A|}M(\sqrt{|A|},n)\Big).\]
\item [{\rm (II)}] If $2\mid n$ and $4\nmid n$, we have
\begin{eqnarray*}
&{}&\mu_{C_1}(J_{N_c,A,a}(D_n/D_1))\\
&=&\dfrac{1}{2}M(|A|,n)-\dfrac{n}{8}(|A|+1)M(|A|,\dfrac{n}{2})+\dfrac{n}{4}\sqrt{|A|}M(\sqrt{|A|},\dfrac{n}{2}).
\end{eqnarray*}
\item [{\rm (III)}] If $4\mid n$, we have
\begin{eqnarray*}
\mu_{C_1}(J_{N_c,A,a}(D_n/D_1))=&\dfrac{1}{2}M(|A|,n)-\dfrac{n}{4}(|A|+1)M(\sqrt{|A|},n).
\end{eqnarray*}
\end{itemize}
The number $\mu_{C_1}(J_{N_c,A,a}(D_n/D_1))$ is the number of all primitive colorings on a regular $n$-gon of $D_n$ with $A$.

\begin{exa}
We consider when $n=4$, $A=\{0,1,2\}$ and $a=0$. Substituting these condition and $d=n$ in (\ref{DR3}), we have 
\begin{eqnarray*}
&{}&\mu_{C_{1},t}(J_{N_c,A,a}(D_4/D_1))\\
&=&-\dfrac{1}{8}(1+2t^2)^2+\dfrac{1}{4}(1+2t^2)^2+\dfrac{1}{4}(1+2t^4)\\
&&+\dfrac{1}{8}(1+2t)^4-\dfrac{1}{4}(1+2t)^2(1+2t^2)-\dfrac{1}{4}(1+2t^2)^2\\
&=&t^2+2t^3.
\end{eqnarray*}
In particular, we have $\mu_{C_1}(J_{N_c,A,a}(D_4/D_1))=3$. All results of placing $3$ colored on the vertices of a regular $4$-gon are described in Figure 5.
\begin{figure}[h]
	\centering
	\includegraphics[width=8cm,height=2.5cm]{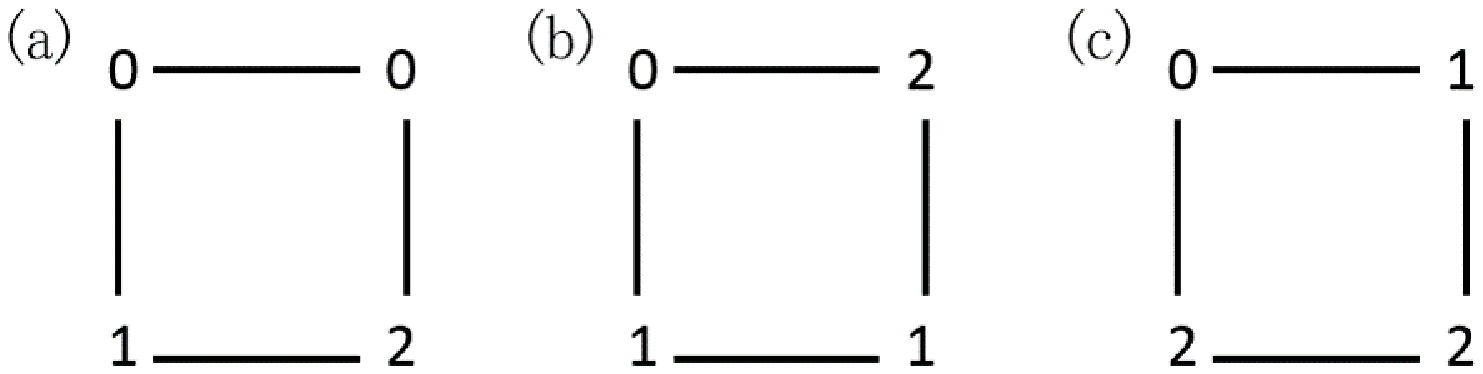}\\
	Figure 5: Primitive colorings on vertices of a regular 4-gon, of $D_4$, with $3$-colors $\{0,1,2\}$.
\end{figure}
In Figure 5, the image (a) is the primitive coloring which has two $0$-colored places, and images (b) and (c) are primitive colorings which have one $0$-colored place.
\end{exa}

\begin{exa}
We consider when $n=5$, $A=\{0,1,2\}$ and $a=0$. Substituting these condition and $d=n$ in (\ref{DR1}), we have 
\begin{eqnarray*}
&{}&\mu_{C_{1},t}(J_{N_c,A,a}(D_5/D_1))\\
&=&-\dfrac{1}{10}(1+2t)^5+\dfrac{1}{2}(1+2t^5)+\dfrac{1}{10}(1+2t)^5-\dfrac{1}{2}(1+2t)(1+2t^2)^2\\
&=&2t^2+4t^3+6t^4.
\end{eqnarray*}
In particular, we have $\mu_{C_1}(J_{N_c,A,a}(D_5/D_1))=12$. All results of placing $3$ colored on the vertices of a regular $5$-gon are described in Figure 6. 
\begin{figure}[h]
	\centering
	\includegraphics[width=9cm,height=6.5cm]{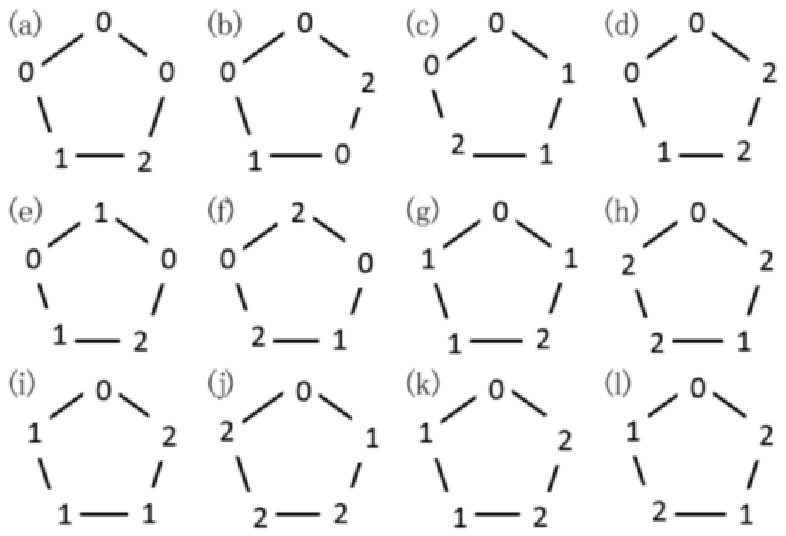}\\
	Figure 6: Primitive colorings on vertices of a regular 5-gon, of $D_5$, with $3$-colors $\{0,1,2\}$.
\end{figure}

In Figure 6, images (a) and (b) are primitive colorings which have three $0$-colored places, and images (c), (d), (e) and (f) are primitive colorings which have two $0$-colored places, and (g), (h), (i), (j), (k) and (l) are primitive colorings which have one $0$-colored place.
\end{exa}

\begin{exa}
We consider when $n=6$, $A=\{0,1\}$ and $a=0$. Substituting these condition and $d=n$ in (\ref{DR2}), we have 
\begin{eqnarray*}
&{}&\mu_{C_{1},t}(J_{N_c,A,a}(D_6/D_1))\\
&=&\dfrac{1}{12}\Big((1+t)^6-(1+t^2)^3-(1+t^3)^2+(1+t^6)\Big)\\
&{}&+\dfrac{1}{4}\Big(-2(1+t^2)(1+t^4)+(1+t)^2(1+t^2)^2+(1+t^2)^3\\
&{}&+2(1+t^6)-(1+t^3)^2-(1+t^6)\Big)=t^3.
\end{eqnarray*}
In particular, we have $\mu_{C_1}(J_{N_c,A,a}(D_6/D_1))=1$. Figure 7 is the unique result which has three $0$-colored places.
\begin{figure}[h]
	\centering
	\includegraphics[width=2.5cm,height=2.5cm]{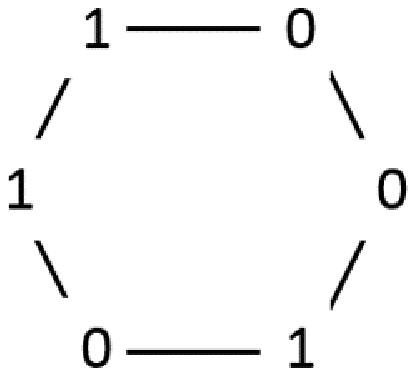}\\
	Figure 7: Primitive colorings on vertices of a regular 6-gon, of $D_6$, with $2$-colors $\{0,1\}$.
\end{figure}
\end{exa}

\newpage
\def\thesection{\Alph{section}}

\setcounter{section}{0}
\section{Appendix}

In this section, we state some notions and propositions which are used in  this paper.

\subsection{For representations of finite groups}
Let $G$ be a finite group.

We define a group homomorphism $\rho:G\rightarrow GL(V)$ by a representation of $G$, where the vector space $V$ is a complex vector space whose dimension is finite. We define the degree of $\rho$ by $\dim_{\mathbb{C}}(V)$. 

The character of $\rho$, which is denoted by $\chi$, is the map $\chi:G\rightarrow\mathbb{C}$ defined by $\chi(g)=\mathrm{Tr}(\rho(g))$ for any $g\in G$.

Next, we define a $G$-set. A $G$-set $X$ is a set equipped with a map 
\[ \iota:G\times X\rightarrow X,\quad \iota(g,x)=gx\]
which satisfies
\[ g_1(g_2x)=(g_1g_2)x,\quad 1x=x\]
for any $g_1,g_2\in G$ and $x\in X$ (Note that we have mentioned it in \S 2.1).

Put $X=\{x_1,\ldots,x_n\}$. We define the permutation representation associated with $X$ by the representation $\rho_X:G\rightarrow GL(V)$ where $V$ is the $n$-dimensional vector space which is formally generated by $\{x_1,\ldots,x_n\}$, and $\rho_X(g)(x_i):=gx_i$ holds for any $g\in G$. Finally, we define the permutation character associated with $X$ by the character of $\rho_X$.

For any $g\in G$, the matrix of $\rho_X(g)$ with respect to the basis $\{x_1,\ldots,x_n\}$ is written by $(x_{i,j})_{1\leq i,j\leq n}$, satisfies 
\[ x_{i,j}=\begin{cases}1 & \mbox{if}\ x_j=gx_i, \\ 0 &\mbox{otherwise.}\end{cases}\]
Let $\chi$ be the permutation character associated with $X$. Then, $\chi(g)$ is the number of $x\in X$ such that $gx=x$ holds. 

\subsection{For commutative rings}

We consider the following two properties on a commutative ring $R$. With the \Mobius\ function, the following proposition holds.

\begin{prop}[\Mobius\ inversion formula]
Let $\vect{a}$ and\\ $\vect{b}$ be infinite vectors of a commutative ring $R$. Two elements $a$ and $b$ satisfy 
\[ b_n=\sum_{d\mid n}a_d\]
for any integer $n\geq 1$ if and only if 
\[ a_n=\sum_{d\mid n}\mu\Big(\dfrac{n}{d}\Big)b_d\]
holds for any integer $n \geq 1$.
\end{prop}

\subsection{For semirings and their ring completion}
In this section, we tell the notion of the ring completion. For more detail, see \cite[\S 9.3]{Hus}.

We define a semiring by a triple $(S,\alpha,\mu)$ where $S$ is a set, the map $\alpha:S\times S\rightarrow S$ is the addition function usually denoted by $\alpha(a,b)=a+b$, and the map $\mu:S\times S\rightarrow S$ is the multiplication function usually denoted by $\mu(a,b)=ab$, and satisfies all axioms of a ring except the existence of negative or additive inverse. For simplicity, we denote a semiring $(S,\alpha,\mu)$ by $S$.

We define a semiring homomorphism from a semiring $S$ to a semiring $S'$ by a map $f:S\rightarrow S'$ such that $f(a+b)=f(a)+f(b)$, $f(ab)=f(a)f(b)$ and $f(0)=0$ hold.

\begin{prop}
For any semiring $S$, there exists a pair $(S^*,\theta)$, where $S^*$ is a ring and $\theta:S\rightarrow S^*$ is a semiring homomorphism such that if for any ring $R$ and a semiring homomorphism $f:S\rightarrow R$, there exists a ring homomorphism $g:S^*\rightarrow R$ such that $g\circ\theta=f$ holds.
\end{prop}

In this paper, we call $(S^*,\theta)$ a ring completion, and we construct the ring completion $(S^*,\theta)$ of a semiring $S$ as follows: We define an equivalence relation on $S\times S$ by that $(a_1,b_1)$ and $(a_2,b_2)$ are equivalent if there exists $c\in S$ such that $a_1+b_2+c=a_2+b_1+c$ holds. We denote the equivalence class of $(a,b)$ by $\circg{a,b}$, and let $S^*$ be the set of all equivalence classes $\circg{a,b}$. The set $S^*$ has the following addition and multiplication which are defined by $\circg{a,b}+\circg{c,d}:=\circg{a+c,b+d}$ and $\circg{a,b}\circg{c,d}:=\circg{ac+bd,ad+bc}$. Next, we define a map $\theta:S\rightarrow S^*$ by $\theta(s)=\circg{s,0}$ for any $s\in S$.

\begin{prop}
Let $\{x_1,\ldots,x_n\}$ be a subset of a semiring $S$. We assume that every elements $\alpha$ of $S$ can be write uniquely 
\[ \alpha=m_1x_1+\cdots+m_nx_n\]
for some integers $m_1,\ldots,m_n\geq 0$. Let $(S^*,\theta)$ be the ring completion of $S$. Then $S^*$ has a $\mathbb{Z}$-basis $\{\theta(x_1),\ldots,\theta(x_n)\}$. 
\end{prop}
\begin{proof}
First, we show that the semiring homomorphism $\theta$ is injective. Let $s_1,s_2\in S$ satisfying $\theta(s_1)=\theta(s_2)$. Then there exists $c\in S$ such that $s_1+c=s_2+c$ holds. By the assumption, we have $s_1=s_2$, that is, the map $\theta$ is injective. From this, elements $\{\theta(x_1),\ldots,\theta(x_n)\}$ are linearly independent. For any $\circg{a,b}\in S^*$, $\circg{a,b}=\theta(a)-\theta(b)$ holds. Hence, $S^*$ has a $\mathbb{Z}$-basis $\{\theta(x_1),\ldots,\theta(x_n)\}$. 
\end{proof}
For simplicity, we write the ring completion of a semiring $S$ as $S^*$, and we write $\theta(a)$ as $a\in S^*$ for any $a\in S$.

Tomoyuki Tamura

Graduate School of Mathematics, Kyushu-University, 

Nishi-ku Fukuoka,  819-0395, Japan.

E-mail: t-tamura@math.kyushu-u.ac.jp
\end{document}